\documentclass[12pt,a4paper,leqno,twoside]{amsart}
\usepackage[a4paper,left=3cm,top=4cm,bottom=4cm,right=3cm]{geometry}
\usepackage[T1]{fontenc}
\usepackage{amssymb,bbm}
\usepackage{amsmath,amsthm,dsfont}
\usepackage{amsfonts,mathrsfs,wasysym}
\usepackage{latexsym}
\usepackage[cp1250]{inputenc}
\usepackage{graphicx}
\usepackage{indentfirst}
\usepackage{enumerate}
\usepackage{geometry}
\usepackage{enumitem}
\usepackage{xparse}
\usepackage{etoolbox}
\usepackage{mathptmx}
\usepackage{fancyhdr}
\usepackage[usenames,dvipsnames]{xcolor}
\usepackage{marvosym}
\usepackage{pifont}

\usepackage{hyperref}
\hypersetup{colorlinks,linkcolor={blue},citecolor={red}}

\patchcmd{\abstract}{\scshape\abstractname}{\textbf{\abstractname}}{}{}

\makeatletter
\DeclareMathAlphabet{\mathcal}{OMS}{cmsy}{m}{n}

\DeclareSymbolFont{operators}{OT1}{ztmcm}{m}{n}
\DeclareSymbolFont{letters}{OML}{ztmcm}{m}{it}
\DeclareSymbolFont{symbols}{OMS}{ztmcm}{m}{n}
\DeclareSymbolFont{largesymbols}{OMX}{ztmcm}{m}{n}
\DeclareSymbolFont{bold}{OT1}{ptm}{bx}{n}
\DeclareSymbolFont{italic}{OT1}{ptm}{m}{it}
\@ifundefined{mathbf}{}{\DeclareMathAlphabet{\mathbf}{OT1}{ptm}{bx}{n}}
\@ifundefined{mathit}{}{\DeclareMathAlphabet{\mathit}{OT1}{ptm}{m}{it}}
\DeclareMathSymbol{\omicron}{0}{operators}{`\o}

\DeclareMathAlphabet{\mathpzc}{OT1}{pzc}{m}{it}
\DeclareSymbolFont{operators}{OT1}{txr}{m}{n}
\SetSymbolFont{operators}{bold}{OT1}{txr}{bx}{n}
\def\operator@font{\mathgroup\symoperators}
\DeclareSymbolFont{italic}{OT1}{txr}{m}{it}
\SetSymbolFont{italic}{bold}{OT1}{txr}{bx}{it}
\DeclareSymbolFontAlphabet{\mathrm}{operators}
\DeclareMathAlphabet{\mathbf}{OT1}{txr}{bx}{n}
\DeclareMathAlphabet{\mathit}{OT1}{txr}{m}{it}
\SetMathAlphabet{\mathit}{bold}{OT1}{txr}{bx}{it}
\DeclareSymbolFont{letters}{OML}{txmi}{m}{it}
\SetSymbolFont{letters}{bold}{OML}{txmi}{bx}{it}
\DeclareFontSubstitution{OML}{txmi}{m}{it}
\DeclareSymbolFont{lettersA}{U}{txmia}{m}{it}
\SetSymbolFont{lettersA}{bold}{U}{txmia}{bx}{it}
\DeclareFontSubstitution{U}{txmia}{m}{it}
\DeclareSymbolFontAlphabet{\mathfrak}{lettersA}
\DeclareSymbolFont{symbols}{OMS}{txsy}{m}{n}
\SetSymbolFont{symbols}{bold}{OMS}{txsy}{bx}{n}
\DeclareFontSubstitution{OMS}{txsy}{m}{n}
\makeatother

\makeatletter
\renewcommand\abstractname{\scshape\bfseries Abstract}

\renewenvironment{proof}[1][\proofname]{\par \pushQED{\qed} \normalfont
  \topsep6\p@\@plus6\p@ \trivlist \itemindent\z@
  \item[\hskip\labelsep\bfseries
    #1\@addpunct{.}]\ignorespaces
}{
  \popQED\endtrivlist\@endpefalse
}

\makeatother

\makeatletter
    \@addtoreset{equation}{section} 
    \renewcommand{\theequation}{{\thesection}.\@arabic\c@equation} 

\makeatother

\makeatletter

\def\section{\@ifstar\unnumberedsection\numberedsection}
\def\numberedsection{\@ifnextchar[
  \numberedsectionwithtwoarguments\numberedsectionwithoneargument}
\def\unnumberedsection{\@ifnextchar[
  \unnumberedsectionwithtwoarguments\unnumberedsectionwithoneargument}
\def\numberedsectionwithoneargument#1{\numberedsectionwithtwoarguments[#1]{#1}}
\def\unnumberedsectionwithoneargument#1{\unnumberedsectionwithtwoarguments[#1]{#1}}
\def\numberedsectionwithtwoarguments[#1]#2{%
  \ifhmode\par\fi
  \removelastskip
  \vskip 4ex\goodbreak
  \refstepcounter{section}%
  \noindent
  \begingroup
  \leavevmode\centering\scshape\bfseries
  \thesection.
  #2
  \par
  \endgroup
  \vskip 1ex\nobreak
  \addcontentsline{toc}{section}{%
    \protect\numberline{\thesection}%
    #1}%
  }
\def\unnumberedsectionwithtwoarguments[#1]#2{%
  \ifhmode\par\fi
  \removelastskip
  \vskip 2ex\goodbreak
  \noindent
  \begingroup
  \leavevmode\centering\scshape\bfseries
  \leavevmode\centering\scshape\bfseries
  #2
  \par
  \endgroup
  \vskip 1ex\nobreak
  \addcontentsline{toc}{section}{%
    #1}%
}
\makeatother

\makeatletter
\def\@seccntformat#1{\csname mythe#1\endcsname}

\let\latex@subsection\subsection
\def\subsection{\@ifstar{\refstepcounter{subsection}\latex@subsection*}{\latex@subsection}}
\def\@makechapterhead#1{%
  \vspace*{40\p@}%
  {\parindent \z@ \raggedright \normalfont
    \interlinepenalty\@M
    \Huge \bfseries #1\par \nobreak
    \vskip 40\p@
  }}
\let\latex@l@chapter\l@chapter
\def\l@chapter#1#2{\begingroup\let\numberline\@gobble\latex@l@chapter{#1}{#2}\endgroup}
\makeatother

\theoremstyle{plain}
\newtheorem{Th}{Theorem}[section]
\newtheorem{Prop}[Th]{Proposition}
\newtheorem{Lem}[Th]{Lemma}
\newtheorem{Cor}[Th]{Corollary}
\theoremstyle{definition}
\newtheorem{Rem}[Th]{Remark}
\newtheorem{Ex}[Th]{Example}

\newtheorem{Def}[Th]{Definition}

\def\bf{\textbf}
\def\it{\textit}
\def\te{\textnormal}
\def\tn{\textnormal}

\def\es{\,{\scriptstyle\te{\ding{60}}}\,}

\def\leq{\leqslant}
\def\geq{\geqslant}
\def\R{{\mathds R}}
\def\N{{\mathds N}}

\def\B{\textit{I\!B}}

\def\D{{\mathrm{dom}}\,}
\def\E{{\mathrm{epi}}\,}

\def\G{{\mathrm{gph}}\,}
\def\I{{\mathrm{int}}\,}
\def\e{\mathsf{e}}
\def\kr{\bar}

\begin{document}
\vspace*{0cm}
\title{Representation of Hamilton-Jacobi equation in optimal control theory with compact control set}
\author{\vspace*{-0.2cm}{Arkadiusz Misztela \textdagger}\vspace*{-0.2cm}}
\thanks{\textdagger\, Institute of Mathematics, University of Szczecin, Wielkopolska 15, 70-451 Szczecin, Poland; e-mail: arkadiusz.misztela@usz.edu.pl, arke@mat.umk.pl}

\begin{abstract} 
In this paper we study the existence of sufficiently regular representations of  Hamilton-Jacobi  equations in optimal control theory with the compact control set.  We introduce a new method to construct  representations for a wide class of Hamiltonians, wider than it was achieved before.  Our result is proved by means of these conditions on Hamiltonian that are necessary for the existence of a representation.  In particular, we solve an open problem  of Rampazzo (2005). We apply the obtained results to reduce a variational problem to an optimal control problem. \\
\vspace{0mm}

\hspace{-1cm}
\noindent  \bf{\scshape Keywords.} Hamilton-Jacobi equations, representations of Hamiltonians, optimal control \\\hspace*{-0.55cm}  theory,  parametrizations of set-valued maps, convex analysis.

\vspace{3mm}\hspace{-1cm}
\noindent \bf{\scshape Mathematics Subject Classification.} 26E25, 49L25, 34A60, 46N10.
\end{abstract}

\maketitle

\pagestyle{myheadings}  \markboth{\small{\scshape Arkadiusz Misztela}
}{\small{\scshape Representations of Hamiltonians}}

\thispagestyle{empty}

\vspace{-1cm}



\section{Introduction}

\noindent The Hamilton-Jacobi equation 
\begin{equation}\label{rowhj}
\begin{array}{rll}
-V_{t}+ H(t,x,-V_{x})=0 &\!\!\textnormal{in}\!\! & (0,T)\times\R^n, \\[0mm]
V(T,x)=g(x) & \!\!\textnormal{in}\!\! & \;\R^n,
\end{array}
\end{equation}
with a convex Hamiltonian $H$  in the gradient variable can be studied with connection to optimal control problem. It is possible, provided there exists  sufficiently regular triple $(A,f,l)$ satisfying the following equality
\begin{equation}\label{hfl}
H(t,x,p)=  \sup_{a\in A}\,\{\,\langle\, p\,,f(t,x,a)\,\rangle\,-\,l(t,x,a)\,\}.
\end{equation}
Then the value function of the Bolza optimal control problem defined by the formula
\begin{equation*}\label{fwfl}
V(t_0,x_0)= \inf_{(x,a)(\cdot)\,\in\, \emph{S}_f(t_0,x_0)}\,\big\{\,g(x(T))+\int_{t_0}^Tl(t,x(t),a(t))\,dt\,\big\}
\end{equation*}
is the unique viscosity solution of \eqref{rowhj}; see, e.g.  \cite{B-CD,C-S-2004,HF,F-P-Rz}, where $\emph{S}_f(t_0,x_0)$ denotes the set of all trajectory-control pairs of the control system
\begin{equation}\label{scs}
\begin{array}{ll}
\dot{x}(t)=f(t,x(t),a(t)),& a(t)\in A\;\;\mathrm{a.e.}
\;\;t\in[t_0,T],\\[0mm]
x(t_0)=x_0.&
\end{array}
\end{equation}
While working with control systems it is usually required from $f$ to be such a function that to every measurable control $a(\cdot)$ on $[t_0,T]$ with values in a compact subset $A$ of $\R^m$ there corresponds a unique solution $x(\cdot)$ of \eqref{scs} defined on $[t_0,T]$. It is guaranteed, for instance, by the local Lipschitz continuity and the sublinear growth of $f$ with respect to~$x$. The local Lipschitz continuity of $l$ with respect to $x$ is also necessary to prove regularities of value functions.

The triple $(A,f,l)$ that satisfies the equation \eqref{hfl} and the conditions stated above is  called a  \it{faithful representation} of  $H$. The use of the name ``faithful representation'' is justified by the fact that there are infinitely many triples $(A,f,l)$, that satisfy the equation \eqref{hfl},  among with there are the ones with totally irregular functions  $f,l$. The triple $(A,f,l)$, not necessarily regular, which satisfies the equality \eqref{hfl}  is called a  \it{representation} of $H$.

The main goal of our paper is to introduce a new method of construction of faithful representations for a wide class of Hamiltonians. This class is wider than the one in the papers \cite{F-S,HI,FR}.  Our result is proved by using only these conditions on  Hamiltonian that are necessary for the existence of a faithful representation. It means that the obtained result is optimal. In particular, we solve an open problem  of Rampazzo \cite[Rem. 2.3]{FR}.

Let the Lagrangian $L$ be the Legendre-Fenchel transform of $H$ in its gradient variable:
\begin{equation}\label{tran1}
 L(t,x,v)= \sup_{p\in\R^{n}}\,\{\,\langle v,p\rangle-H(t,x,p)\,\}.
\end{equation}
Here $\langle v,p\rangle$ denotes the inner product of $v$ and $p$. It is possible for $L$ to attain the value $+\infty$. The sets: $\D\varphi=\{\,x\in\R^n\mid\varphi(x)\not=\pm\infty\,\}$, 
$\G\varphi=\{\,(x,r)\in\R^n\times\R\mid\varphi(x)=r\,\}$ and $\E\varphi=\{\,(x,r)\in\R^n\times\R\mid\varphi(x)\leq r\,\}$ are called the \emph{effective domain}, the \it{graph} and the \it{epigraph} of $\varphi$, respectively. 

In 1985 Ishii \cite{HI} proposed a representation $(A,f,l)$ involving continuous functions $f,l$ with the infinite-dimensional control set $A$ and expressed the solution of a stationary Hamilton-Jacobi equation as the value function of an associated infinite horizon optimal control problem. The lack of local Lipschitz continuity of functions $f,l$ with respect to the variable $x$ in Ishii \cite{HI} paper causes   a lot of trouble in applications. Moreover, in general, not to every control $u(\cdot)$ there corresponds exactly one trajectory $x(\cdot)$. This means that one can not control  the system completely by selecting one of controls.

In 2005 Rampazzo \cite{FR} constructed a faithful representation by using set-valued  and convex analysis. His representation $(A,f,l)$ of $H$ is a graphical representation, i.e. a triple $(A,f,l)$ which satisfies  $\e(t,x,A)=\G L(t,x,\cdot)$, where $\e=(\,f,l\,)$. Examples \ref{prz-rep-111} and~\ref{prz-rep-112} show that a graphical representation is a faithful representation, if the Lipschitz-type condition on  $L(t,\cdot,\cdot)$ is assumed. It is a strong assumption, because  $L(t,\cdot,\cdot)$ is usually a lower semicontinuous function (see Exs. \ref{ex-1}, \ref{ex-2}, \ref{ex-4}). Such strong condition in \cite{FR} is the condition (H5). This problem was also noticed by  Rampazzo (see \cite[Rem. 2.3]{FR}).

In 2014 Frankowska-Sedrakyan \cite{F-S} investigated faithful representations 
of Hamiltonians that are measurable with respect to the time variable. In this case 
 Lipschitz constants of  Hamiltonians should depend on time. Frankowska-Sedrakyan \cite{F-S} noticed that if Lipschitz constants of  Hamiltonians are measurable functions, then the 
results of  Rampazzo \cite{FR} do not allow to claim whether Lipschitz constants of 
faithful representations of these  Hamiltonians are also measurable. It is well-known that 
in applications one requires not only measurability of  Lipschitz constants of faithful representations but also integrability.  This problem was solved by 
 Frankowska-Sedrakyan \cite{F-S} by indicating the precise Lipschitz constants of faithful 
representations depending on  Lipschitz constants of Hamiltonians. Besides, they studied 
stability of faithful representations. This result allowed 
Sedrakyan \cite{HS} to prove  appropriate convergence of value functions. However,
Frankowska-Sedrakyan \cite{F-S} used a graphical representation similarly to Rampazzo \cite{FR}. Therefore, they also need such strong condition (see \cite[(H5)]{F-S}).

 In this paper we solve the above problem concerning a graphical representation from  \cite{F-S, FR}. To this end, we introduced a new method of construction of a  faithful representation. Our representation $(A,f,l)$ of $H$ is an epigraphical representation, i.e. a triple $(A,f,l)$ which satisfies the condition $\G L(t,x,\cdot)\subset\e(t,x,A)\subset\E L(t,x,\cdot)$, where $\e=(\,f,l\,)$. An epigraphical representation is constructed by parametrizing  $\E L(t,x,\cdot)$ instead of  $\D L(t,x,\cdot)$ as in the case of  graphical representation. It implies that the dimension of the control set in our construction increases by one comparing to the graphical construction. The set  $\E L(t,x,\cdot)$ is not bounded as opposed to the set $\D L(t,x,\cdot)$. This fact causes new difficulties, but  we are able to deal with them. Thus, we obtain results that do not need such strong assumptions as in papers \cite{F-S,FR}. Besides, we indicate precise  Lipschitz constants of faithful representations similarly to Frankowska-Sedrakyan \cite{F-S}. In particular,\linebreak our results imply the stability of representations. In Subsection \ref{ncfr} we show that not every Hamiltonian has a faithful representation with the 
compact control set. This property is satisfied if  Lagrangian is bounded 
on the effective domain (see Thm. \ref{podr22_th_wk}). Moreover, our construction of a 
faithful representation can be applied to Hamiltonians with unbounded  Lagrangians on effective domains. However, in that case we obtain faithful representations with the unbounded 
control set. Such results will be contained in  \cite{AM}.

We apply our results  to reduce a variational problem to on optimal control problem (see Subsect.~\ref{rvptoocp}). More precisely, let us consider a variational problem associated with the given Lagrangian $L$. Let us define Hamiltonian $H$ as  the Legendre-Fenchel transform of $L$ in its velocity variable. Applying our result to Hamiltonian $H$ we obtain its faithful\linebreak representation $(A,f,l)$. Then the variational  problem associated with Lagrangian $L$ is~equi\-valent to the optimal control problem associated with the triple $(A,f,l)$ (see Thm. \ref{thm-reduct}).
Ealier, Olech \cite{CO-69} and Rockafeller \cite{RTR73,RTR} investigated the opposite 
problem that is a reduction of an  optimal control problem to a variational problem. 
More precisely, they considered the optimal control problem associated with the given triple 
$(A,f,l)$. Using this triple they defined Lagrangian $L$ in such a way that 
the optimal control problem associated with the triple  $(A,f,l)$ is equivalent 
to the variational problem associated with Lagrangian~$L$. The details concerning this reduction can also be found in the comprehensive monograph of Clarke \cite{FC}. Therefore, the above results 
prove that there exists strong correlation between variational problems and optimal control problems.

The outline of the paper is as follows. Section \ref{section-2} contains hypotheses and preliminary results. In Section~\ref{section-3} we gathered our main results. Sections \ref{wk-kon-istr}, \ref{pofrepth}, \ref{thms-stab}, \ref{thm-reduct-sect} contain proofs.


\section{Hypotheses and preliminary results}\label{section-2}

\noindent  We shall consider the following assumptions on the Hamiltonian:

\begin{enumerate}[leftmargin=9.7mm]
\item[\te{\bf{(H1)}}] $H:[0,T]\times\R^{n}\times\R^{n}\rightarrow\R$
is Lebesgue measurable  in $t$ for any $x,p\in\R^n$;
\item[\te{\bf{(H2)}}] $H(t,x,p)$ is continuous with respect to $(x,p)$ for every $t\in[0,T]$;
\item[\te{\bf{(H3)}}] $H(t,x,p)$ is convex with respect to $p$ for every $(t,x)\in[0,T]\times\R^n$;
\item[\te{\bf{(H4)}}] There exists a measurable map $c:[0,T]\to[0,+\infty)$ such that for every\\ 
\hspace*{-8.5mm}$t\in[0,T]$ and $x,p,q\in\R^n$ one has $|H(t,x,p)-H(t,x,q)|\leq c(t)(1+|x|)|p-q|$.
\end{enumerate}

 An extended-real-valued function is called \it{proper} if it never attains  the value $-\infty$ and it is not identically equal to $+\infty$. If  $H(t,x,\cdot)$ is proper, convex and lower semicontinuous for each $(t,x)$, then  $L(t,x,\cdot\,):=H^{\ast}(t,x,\cdot\,)$, where $^{\ast}$ denotes the Legendre-Fenchel transform, also has  these properties. Moreover, the following equality $H(t,x,\cdot\,)=L^{\ast}(t,x,\cdot\,)$ holds, cf. \cite[Thm. 11.1]{R-W}. By means of the properties of the Legendre-Fenchel transform from \cite{R-W} we can prove an equivalent version of (H1)$-$(H4) in the Lagrangian terms:

\begin{Prop}\label{prop2-fmw} Assume that $H$ satisfies \te{(H1)$-$(H3)}. If $L(t,x,\cdot\,)=H^{\ast}(t,x,\cdot\,)$, then
\begin{enumerate}
\item[\te{\bf{(L1)}}] $L:[0,T]\times\R^{n}\times\R^{n}\rightarrow\R\cup\{+\infty\}$ is Lebesgue-Borel-Borel measurable\te{;}
\item[\te{\bf{(L2)}}] $L(t,x,v)$  is  lower semicontinuous  with respect to $(x,v)$ for every $t\in[0,T]$\te{;}
\item[\te{\bf{(L3)}}] $L(t,x,v)$ is convex and proper with respect to $v$ for every $(t,x)\in[0,T]\times\R^n$\te{;}
\item[\te{\bf{(L4)}}] $\forall\,(t,x,v)\in[0,T]\times\R^n\times\R^n\;\;\forall\,x_i\rightarrow x\;\;
\exists\,v_i\rightarrow v\;:\;L(t,x_i,v_i)\rightarrow L(t,x,v)$\te{;}
\item[]\hspace{-1.3cm}Additionally, if $H$ satisfies \te{(H4)}, then
\item[\te{\bf{(L5)}}] $\forall\,(t,x,v)\in[0,T]\times\R^n\times\R^n \;:\; |v|>c(t)(1+|x|)\;\Rightarrow\; L(t,x,v) =+\infty$\te{;}
\item[]\hspace{-1.3cm}Additionally, if $H$ is continuous, then $L$ is lower semicontinuous and
\item[\te{\bf{(L6)}}] $\forall\,(t,x,v)\in[0,T]\times\R^n\times\R^n\;\;\forall\,(t_i,x_i)\rightarrow (t,x)\;\;\exists\,v_i\rightarrow v\;:\;L(t_i,x_i,v_i)\rightarrow L(t,x,v)$.
\end{enumerate}
\end{Prop}

Actually, we can prove that (H1)$-$(H4) are equivalent to  (L1)$-$(L5). 

Let us define the set-valued map $E_L:[0,T]\times\R^n\multimap\R^n\times\R$ by the following formula
\begin{equation*}
E_L(t,x):= \{\,(v,\eta)\in\R^n\times\R\,\mid\, L(t,x,v)\leq \eta\,\}.
\end{equation*}

 We say that a  set-valued map $F:[0,T]\multimap\R^m$ is \it{measurable}, if for every
open set $U\subset\R^m$ the inverse image  $F^{-1}(U):= \{\,t\in[0,T]\mid F(t)\cap U\not=\emptyset\,\}$
is a Lebesgue measurable set. The conditions (L1)$-$(L3) imply that a set-valued map $t\to E_L(t,x)$ is measurable for every $x\in\R^n$ and the set $E_L(t,x)$ is nonempty, closed and convex for all $(t,x)\in[0,T]\times\R^n$. 
 
The set $\G F:= \{\,(z,y)\mid y\in F(z)\,\}$ is called a \it{graph} of the set-valued map $F$.  From (L2) it follows that a set-valued map $x\to E_L(t,x)$ has a closed graph in $\R^n\times\R$ for all $t\in[0,T]$.

We say that a set-valued map $F:\R^n\multimap\R^m$ is \it{lower semicontinuous} in  Kuratowski's sense, if for every open set $U\subset\R^m$ the set $F^{-1}(U)$ is open. It is equivalent to the following condition: $\forall\,(z,y)\in\G F\;\;\forall\;z_i\rightarrow z\;\;\exists\;y_i\rightarrow y\;:\; y_i\in F(z_i)$ for all large $i\in\N$. 
The condition (L4) means that a set-valued map $x\to E_L(t,x)$ is lower semicontinuous in  Kuratowski's sense for every $t\in[0,T]$. 

For a nonempty subset $K$ of $\R^n$ we define $\|K\|:=\sup_{x\in K}|x|$. The condition (L5) implies that  $\|\D L(t,x,\cdot)\|\leq c(t)(1+|x|)$ for every $(t,x)\in[0,T]\times\R^n$. 

If  $L$ is lower semicontinuous with respect to all variables and satisfies (L6) then the set-valued map $E_L$ has a closed graph and is lower semicontinuous. 

If we combine the above facts we obtain the following corollary:

\begin{Cor}\label{wrow-wm}
Assume that $H$ satisfies \te{(H1)$-$(H3)}. If $L(t,x,\cdot\,)=H^{\ast}(t,x,\cdot\,)$, then
\begin{enumerate}
\item[\te{\bf{(E1)}}] $E_L(t,x)$ is a nonempty, closed, convex subset of $\;\R^{n+1}$ for all $(t,x)\in[0,T]\times\R^n$\te{;}
\item[\te{\bf{(E2)}}] $x\to E_L(t,x)$ has a closed graph for every $t\in[0,T]$\te{;}
\item[\te{\bf{(E3)}}] $x\to E_L(t,x)$ is lower semicontinuous for every  $t\in[0,T]$\te{;}
\item[\te{\bf{(E4)}}] $t\to E_L(t,x)$ is measurable for every $x\in\R^n$\te{;}
\item[]\hspace{-1.3cm}Additionally, if $H$ satisfies \te{(H4)}, then
\item[\te{\bf{(E5)}}] $\|\D L(t,x,\cdot)\|\leq c(t)(1+|x|)$ for every $(t,x)\in[0,T]\times\R^n$\te{;}
\item[]\hspace{-1.3cm}Additionally, if  $H$ is continuous, then
\item[\te{\bf{(E6)}}] $(t,x)\to E_L(t,x)$ has a closed graph and is lower semicontinuous.
\end{enumerate}
\end{Cor}

\subsection{Lipschitz set-valued map $\pmb{x\to E_L(t,x)}$}
In this subsection we present Hausdorff continuity of a set-valued map in  Lagrangian and Hamiltonian terms. Let $\B(\kr{x},R)$ denote the closed ball in $\R^n$ of center $\kr{x}$ and radius $R\geq 0$. We
set $\B_R:=\B(0,R)$ and $\B:=\B(0,1)$. 
\begin{Th}\label{tw2_rlhmh}
Assume \te{(H1)$-$(H3)}. Let  $L(t,x,\cdot\,)=H^{\ast}(t,x,\cdot\,)$ and  $H(t,x,\cdot\,)=L^{\ast}(t,x,\cdot\,)$. Then there are the equivalences $\te{(HLC)}\Leftrightarrow\te{(LLC)}\Leftrightarrow\te{(ELC)}$\tn{:}

\tn{\bf{(HLC)}} For any $R>0$ there exists a measurable map $k_R:[0,T]\to[0,+\infty)$ such that $|\,H(t,x,p)-H(t,y,p)\,|\,\leq\, k_R(t)\,(1+|p|)\,|x-y|$ for all $t\in[0,T]$, $x,y\in\B_R$, $p\in\R^n$.

\tn{\bf{(LLC)}} For any $R>0$ there exists a measurable map $k_R:[0,T]\to[0,+\infty)$  such that for all $t\in[0,T]$, $x,y\in \B_R$, $v\in\D L(t,x,\cdot)$ there exists $u\in\D L(t,y,\cdot)$ satisfying inequalities $|u-v|\leq k_R(t)|y-x|$ and $L(t,y,u)\leq L(t,x,v)+k_R(t)|y-x|$.

\tn{\bf{(ELC)}} For any $R>0$ there exists a measurable map $k_R:[0,T]\to[0,+\infty)$ such that 
$E_L(t,x)\,\subset\, E_L(t,y)+k_R(t)\,|x-y|\,(\B\times[-1,1])$ for all $t\in[0,T]$, $x,y\in\B_R$.

Equivalences hold for the same map $k_R(\cdot)$.
\end{Th}

 Theorem \ref{tw2_rlhmh} follows from Propositions \ref{0rwwl1} and \ref{0lnrwwl1} that are proven below. 

Let $K$ be a nonempty subset of $\R^m$. The distance from $x\in\R^m$ to $K$ is defined by $d(x,K):=\inf_{y\in K}|x-y|$. For nonempty subsets $K$ and $D$ of $\R^m$, the extended Hausdorff distance between $K$ and $D$ is defined by
\begin{equation}
\mathscr{H}(K,D):= \max\big\{\,\sup_{x\in K}d(x,D),\;\sup_{x\in D}d(x,K)\,\big\}\in\R\cup\{+\infty\}.
\end{equation}

By Theorem \ref{tw2_rlhmh} (ELC) we obtain the following corollary:

\begin{Cor}\label{hlc-cor-ner}
Assume that  $H$ satisfies  \te{(H1)$-$(H3)} and \tn{(HLC)}. If $L(t,x,\cdot\,)=H^{\ast}(t,x,\cdot\,)$, then the following inequality
\begin{equation}\label{n2_hd}
\mathscr{H}(E_L(t,x),E_L(t,y))\leq 2k_R(t)\,|x-y|
\end{equation}
holds for any $t\in[0,T]$, $x,y\in \B_R$ and $R>0$.
\end{Cor}

The epi-sum of functions $\phi,\,\psi:\R^n\rightarrow\R\cup\{+\infty\}$ is a
function $\phi\es\psi:\R^n\rightarrow\R\cup\{\pm\infty\}$
given by the formula
\begin{equation*}
(\phi\es\psi)(v):= \inf_{u\in\R^n}\{\phi(u)+\psi(v-u)\}.
\end{equation*}
Let functions $\phi,\,\psi:\R^n\rightarrow\R\cup\{+\infty\}$ be  proper,
convex and lower semicontinuous. We also assume that  $\D\psi=\R^n$.
Then the epi-sum $\phi^{\ast}\es\, \psi^{\ast}$ is a proper, convex and lower semicontinuous function. Moreover,  the following equality holds, cf. \cite[Thm. 11.23]{R-W},
\begin{equation}\label{0epi-suma-row}
(\phi+\psi)^{\ast}=\phi^{\ast}\es\, \psi^{\ast}.
\end{equation}

\begin{Prop}\label{0rwwl1}
Assume that  $p\rightarrow H(t,x,p)$ and $p\rightarrow H(t,y,p)$ are two  real-valued convex functions. Assume further that  $L(t,x,\cdot\,)=H^{\ast}(t,x,\cdot\,)$ and  $L(t,y,\cdot\,)=H^{\ast}(t,y,\cdot\,)$. Then the following conditions are equivalent:
\begin{enumerate}
\item[$\pmb{\te{(a)}}$] $H(t,x,p)\leq H(t,y,p)+k_R(t)\,(1+|p|)\,|x-y|$\, for all $p\in\R^n$.
\vspace{1mm}
\item[$\pmb{\te{(b)}}$] For all\, $v\in\D L(t,x,\cdot)$\, there exists\, $u\in\D L(t,y,\cdot)$\, such that  $|u-v|\leq k_R(t)\,|x-y|$\, and\,  $L(t,y,u)\leq L(t,x,v)+k_R(t)\,|x-y|$.
\end{enumerate}
\end{Prop}

\begin{proof}
We start with the proof of  implication  $\te{(a)}\Rightarrow\te{(b)}$. Let $\psi(p):=k_R(t)\,(1+|p|)\,|x-y|$ and $\phi(p):=H(t,y,p)$  for every $p\in\R^n$. It is not difficult to calculate that for every $v\in\R^n$ 
\begin{equation}\label{0prop-wa2}
\psi^{\ast}(v)=\left\{
\begin{array}{lll}
\:-k_R(t)\,|x-y| & \te{if} & |v|\leq k_R(t)\,|x-y| \\
\:+\infty & \te{if} & |v|> k_R(t)\,|x-y|.
\end{array}
\right.
\end{equation}
 We notice that the function  $\psi$ is proper, convex, lower semicontinuous and $\D\psi=\R^n$. Therefore, by the equality (\ref{0epi-suma-row}) it follows that for every $v\in\R^n$
\begin{equation}\label{0prop-wa1}
(\phi+\psi)^{\ast}(v)=(\phi^{\ast}\es \psi^{\ast})(v)=\inf_{u\,:\,|v-u|\,\leq\, k_R(t)\,|x-y|}\left\{\,L(t,y,u)-k_R(t)\,|x-y|\,\right\}.
\end{equation}
The inequality $\te{(a)}$ implies that $H(t,x,p)\leq \phi(p)+\psi(p)$ for every $p\in\R^n$. Therefore,
by the property of the Legendre-Fenchel transform we obtain  $(\phi+\psi)^{\ast}(v)\leq L(t,x,v)$ for all $v\in\R^n$. By the equality (\ref{0prop-wa1}) we get for all $v\in\R^n$
\begin{equation}\label{0prop-wa4}
L(t,x,v)\geq \inf_{u\,:\,|v-u|\,\leq\, k_R(t)\,|x-y|}\{\,L(t,y,u)-k_R(t)\,|x-y|)\,\}.
\end{equation}
The function  $u\rightarrow L(t,y,u)-k_R(t)\,|x-y|$ is proper and lower semicontinuous, so it achieves its minimum
on the compact set $\{\,u\,\mid\,|v-u|\leq k_R(t)\,|x-y|\,\}$. Using the inequality  (\ref{0prop-wa4}), we obtain the condition $\te{(b)}$. This completes the proof $\te{(a)}\Rightarrow\te{(b)}$.

Now, we prove the implication $\te{(b)}\Rightarrow\te{(a)}$. To this end, we fix $\kr{p}\in\R^n$ and $\varepsilon>0$. Because of $H(t,x,\cdot\,)=L^{\ast}(t,x,\cdot\,)$,  there exists $\kr{v}\in\D L(t,x,\cdot)$ such that
\begin{equation}\label{00prop-wa5}
    H(t,x,\kr{p})-\varepsilon\leq\langle \kr{p},\kr{v}\rangle-L(t,x,\kr{v}).
\end{equation}
By the condition $\te{(b)}$ there exists $\kr{u}\in\D L(t,y,\cdot)$ such that
\begin{equation}\label{00prop-wa6}
    |\kr{u}-\kr{v}|\leq k_R(t)\,|y-x| \;\;\te{and}\;\; L(t,y,\kr{u})\leq L(t,x,\kr{v})+k_R(t)\,|y-x|.
\end{equation}
By the inequalities (\ref{00prop-wa5}) and (\ref{00prop-wa6}) we obtain
\begin{eqnarray*}
    H(t,x,\kr{p})-\varepsilon  &\leq &  \langle \kr{p},\kr{v}\rangle-L(t,x,\kr{v})+ H(t,y,\kr{p})-\langle \kr{p},\kr{u}\rangle+L(t,y,\kr{u})\\
    &\leq & H(t,y,\kr{p})+ |\kr{p}|\,|\kr{v}-\kr{u}|+L(t,y,\kr{u})-L(t,x,\kr{v})\\
    &\leq & H(t,y,\kr{p})+k_R(t)\,(1+|\kr{p}|)\,|x-y|.
\end{eqnarray*}
As $\varepsilon>0$ is an arbitrary number, we get $H(t,x,\kr{p}) \leq H(t,y,\kr{p})+k_R(t)\,(1+|\kr{p}|)\,|x-y|$. Also, $\kr{p}\in\R^n$ is arbitrary, so we have the inequality $H(t,x,p) \leq H(t,y,p)+k_R(t)\,(1+|p|)\,|x-y|$ for every $p\in\R^n$. It complete the proof.
\end{proof}

\begin{Prop}\label{0lnrwwl1}
Assume that   $v\rightarrow L(t,x,v)$ and $v\rightarrow L(t,y,v)$ are two proper extended-real-valued functions.  Then the following conditions are equivalent:
\begin{enumerate}
\item[$\pmb{\te{(a)}}$] For all\, $v\in\D L(t,x,\cdot)$\, there exists\, $u\in\D L(t,y,\cdot)$\, such that $|u-v|\leq k_R(t)\,|y-x|$\, and\,  $L(t,y,u)\leq L(t,x,v)+k_R(t)\,|y-x|$.
\vspace{1mm}
\item[$\pmb{\te{(b)}}$] $E_L(t,x)\,\subset\, E_L(t,y)+k_R(t)\,|x-y|\,(\B\times[-1,1])$.
\end{enumerate}
\end{Prop}

\begin{proof}
We start with the proof of  implication $\te{(a)}\Rightarrow\te{(b)}$. Without loss of generality
we assume that $x\not=y$. Let $(v,\eta)\in E_L(t,x)$. Then $L(t,x,v)\leq \eta$. So $v\in\D L(t,x,\cdot)$.
By the condition $\te{(a)}$, there exists $u\in\D L(t,y,\cdot)$ such that
\begin{equation*}
\te{(i)}\;\;|u-v|\leq k_R(t)\,|y-x| \;\;\te{and}\;\;\te{(ii)}\;\; L(t,y,u)\leq L(t,x,v)+k_R(t)\,|y-x|.
\end{equation*}
Let us define  $\mu:=\eta+k_R(t)\,|y-x|$, $s:=-1$ and
\begin{equation*}
b:=(v-u)/(k_R(t)\,|y-x|)\;\;\te{if}\;\;k_R(t)>0,\hspace{0.5cm}b:=0\;\;\te{if}\;\;k_R(t)=0.
\end{equation*}
The inequality (i) implies that $b\in\B$. Besides, from (ii) we obtain 
\begin{equation*}
L(t,y,u)\leq L(t,x,v)+k_R(t)\,|y-x|\leq\eta+k_R(t)\,|y-x|=\mu.
\end{equation*}
Therefore, $(b,s)\in\B\times[-1,1]$ and $(u,\mu)\in\E L(t,y,\cdot)=E_L(t,y)$. So, we get
\begin{eqnarray*}
(v,\eta) &=& (u,\mu)+k_R(t)\,|y-x|\,(b,s)\\
&\in & E_L(t,y)+k_R(t)\,|x-y|\,(\B\times[-1,1]).
\end{eqnarray*}
Thus, the condition  $\te{(b)}$ of the proposition is proven.

Now, we prove the implication $\te{(b)}\Rightarrow\te{(a)}$. Let $v\in\D L(t,x,\cdot)$. Then
$(v,L(t,x,v))\in E_L(t,x)$. Therefore, by the condition  $\te{(b)}$ we obtain
$$(v,L(t,x,v))\in  E_L(t,y)+k_R(t)\,|x-y|\,(\B\times[-1,1]).$$
So, there exists $(u,\mu)\in E_L(t,y)$ and $(b,s)\in\B\times[-1,1]$ such that
\begin{eqnarray}\label{00arfr-1}
(v,L(t,x,v))=(u,\mu)+k_R(t)\,|y-x|\,(b,s).
\end{eqnarray}
Because of $(u,\mu)\in E_L(t,y)$,  $L(t,y,u)\leq \mu$. Hence  $u\in\D L(t,y,\cdot)$. By the equality (\ref{00arfr-1}) we have  $|u-v|= k_R(t)\,|y-x|\,|b|\leq k_R(t)\,|y-x|$ and
\begin{eqnarray*}
L(t,y,u)\;\;\leq\;\;\mu &=& L(t,x,v)+k_R(t)\,|y-x|(-s)\\
&\leq & L(t,x,v)+k_R(t)\,|y-x|.
\end{eqnarray*}
 Thus, we have proven that for every $v\in\D L(t,x,\cdot)$ there exists  $u\in\E L(t,y,\cdot)$ such that $|u-v|\leq k_R(t)\,|y-x|$ and  $L(t,y,u)\leq L(t,x,v)+k_R(t)\,|y-x|$.
It completes the proof of the proposition.
\end{proof}

\subsection{Examples of Hamiltonians}\label{przy-podroz} 
In this subsection we present  examples of Hamiltonians which satisfy (H1)$-$(H4) and (HLC). These examples have nonregular Lagrangians, so they do not fulfill conditions of theorems contained in  \cite{F-S,FR}. 

\begin{Ex}\label{ex-1}
Let us define the Hamiltonian $H:\R\times\R\rightarrow\R$ by the formula
\begin{equation*}
H(x,p):=\max\{\,|p|\,|x|-1,0\,\}.
\end{equation*}
This Hamiltonian satisfies conditions  (H1)$-$(H4) and (HLC). The Lagrangian  $L:\R\times\R\rightarrow\R\cup\{+\infty\}$ given by the formula (\ref{tran1}) has the form
\begin{equation*}
L(x,v)=\left\{
\begin{array}{ccl}
+\infty & \te{if} & v\not\in[-|x|,|x|\,],\;x\not=0,\\[1mm]
\left|\frac{\displaystyle v}{\displaystyle x}\right| & \te{if} & v\in[-|x|,|x|\,],\;x\not=0, \\[1mm]
0 & \te{if} & v=0,\; x=0,\\[0mm]
+\infty & \te{if} & v\not=0,\;x=0.
\end{array}
\right.
\end{equation*}
Obviously,  $\D L(x,\cdot)=[-|x|,|x|\,]$ for all $x\in\R$. Moreover, the function  $(x,v)\rightarrow L(x,v)$ does not satisfy the assumption (H5) of~\cite{F-S,FR}. Indeed, it is not continuous on the set $\D L$, because $\lim_{i\rightarrow\infty}L\left(1/i,1/i\right)=1\not= 0=L(0,0)$.
\end{Ex}

\begin{Ex}[Rampazzo]\label{ex-2}
Let us define the Hamiltonian $H:\R\times\R\rightarrow\R$ by the formula
\begin{equation*}
H(x,p):=\sqrt{1+p^2}-|x|.
\end{equation*}
This Hamiltonian satisfies conditions (H1)$-$(H4) and (HLC). The Lagrangian  $L:\R\times\R\rightarrow\R\cup\{+\infty\}$ given by the formula (\ref{tran1})
has the following form
\begin{equation*}
L(x,v)=\left\{
\begin{array}{ccl}
|x|-\sqrt{1-v^2} & \te{if} & v\in[-1,1],\\[1mm]
+\infty & \te{if} & v\not\in[-1,1].
\end{array}
\right.
\end{equation*}
Obviously, $\D L(x,\cdot)=[-1,1]$ for all $x\in\R$. We notice that the function $(x,v)\rightarrow L(x,v)$ is continuous on the set $\D L$, but it does not fulfill the condition (H5) of \cite{F-S,FR}.
\end{Ex}

\begin{Ex}\label{ex-4}
Let us define the Hamiltonian $H:\R\times\R\rightarrow\R$ by the formula
\begin{equation*}
H(x,p):=\left\{
\begin{array}{ccl}
(\sqrt{|xp|}-1)^2 & \te{if} & |xp|> 1, \\[1mm]
0 & \te{if} & |xp|\leq 1.
\end{array}
\right.
\end{equation*}
This Hamiltonian satisfies conditions (H1)$-$(H4) and (HLC). The Lagrangian  $L:\R\times\R\rightarrow\R\cup\{+\infty\}$ given by the formula (\ref{tran1})
has the following form
\begin{equation*}
L(x,v)=\left\{
\begin{array}{ccl}
+\infty & \te{if} & v\not\in(-|x|,|x|\,),\;x\not=0,\\[1mm]
\frac{\displaystyle |v|}{\displaystyle |x|-|v|} & \te{if} & v\in(-|x|,|x|\,),\;x\not=0, \\[2mm]
0 & \te{if} & v=0,\; x=0,\\[0mm]
+\infty & \te{if} & v\not=0,\;x=0.
\end{array}
\right.
\end{equation*}
The set $\D L(x,\cdot)=(-|x|,|x|\,)$  is not closed and the function $v\rightarrow L(x,v)$ is not bounded on this set for every $x\in\R\setminus\{0\}$. Moreover, the function  $(x,v)\rightarrow L(x,v)$  is not continuous on the set $\D L$. 
\end{Ex}


\section{Main Results}\label{section-3}
\noindent In this section we describe main results of the paper that concern  faithful representations with the compact control set. We start with proving that representations are not determined uniquely. In addition to this, they can be totally irregular.

We consider the  Hamiltonian $H:\R\times\R\rightarrow\R$ given by the formula  $H(x,p):=|p|$.\linebreak We notice that the triple $([-1,1],f,l)$ is a representation of this  Hamiltonian if functions $f,l:\R\times [-1,1]\rightarrow\R$ satisfy the following conditions:
\begin{equation}\label{wardlap}
|f(x,a)|\leq 1,\; f(x,1)=1,\;f(x,-1)=-1\;\;\;\te{and}\;\;\; l(x,a)\geq 0,\; l(x,1)=l(x,-1)=0.
\end{equation}
Let $i(\cdot)$ and $j(\cdot)$ be arbitrary functions on $\R$ with values in  $[0,\infty)$.
Then  functions
\begin{equation}\label{pfcn}
f_i(x,a):=a\,(1+|a|\,i(x))/(1+i(x)),\;\; l_j(x,a):=(1-|a|)\,j(x),\;\; x\in\R,\;\;a\in[-1,1]
\end{equation}
satisfy conditions \eqref{wardlap}. Therefore, every triple $([-1,1],f_i,l_j)$, where $f_i,l_j$ are given by \eqref{pfcn}, is a representation of the Hamiltonian $H(x,p)=|p|$. There  also exist representations with nonmeasurable (with respect to the state variable) functions $f_i,l_j$, for instance if $i(\cdot)$ and $j(\cdot)$ are not measurable.
However, our results show that from  the set of representations one can always choose a faithful representation.

\subsection{Necessary condition for the existence of a faithful representation}\label{ncfr}
We start this subsection with introducing the condition for an upper bound of the Lagrangian on its  effective domain.

\vspace{1.5mm}
\tn{\bf{(BLC)}} There exists a map $\lambda:[0,T]\times\R^n\rightarrow\R$ measurable in $t$ for every $x\in\R^n$ and continuous in $x$ for every $t\in[0,T]$ such that $L(t,x,v)\leq\lambda(t,x)$ for every $(t,x)\in[0,T]\times\R^n$ and $v\in\D L(t,x,\cdot)$. Assume further that for any $R>0$ there exists a measurable map $k_R:[0,T]\to[0,+\infty)$ such that $\lambda(t,\cdot)$ is $k_R(t)$-Lipschitz on $\B_R$ for every $t\in[0,T]$.

\begin{Th}\label{podr22_th_wk}
Let $A$ be a nonempty compact set. We suppose that  $f:[0,T]\times\R^n\times A\rightarrow\R^n$ and  $l:[0,T]\times\R^n\times A\rightarrow\R$  are  measurable in $t$ for all $(x,a)\in\R^n\times A$ and continuous in $(x,a)$ for all $t\in[0,T]$. Furthermore, we assume that for every $R>0$ there exists a measurable map  $k_R:[0,T]\to[0,+\infty)$ such that $l(t,\cdot,a)$ is $k_R(t)$-Lipschitz on $\B_R$ for every $t\in[0,T]$ and $a\in A$. If the triple $(A,f,l)$ is a representation of $H$, then $L(t,x,\cdot\,):=H^{\ast}(t,x,\cdot\,)$  satisfies the condition \te{(BLC)} with the same map $k_R(\cdot)$. Moreover, if $\,f,l\,$ are continuous, then $\,\lambda\,$ is also continuous.
\end{Th}

The proof of Theorem  \ref{podr22_th_wk} is given in Section \ref{wk-kon-istr}.

\begin{Cor}
Let $A$ be a nonempty compact set.  Assume that  $f:[0,T]\times\R^n\times A\rightarrow\R^n$ and  $l:[0,T]\times\R^n\times A\rightarrow\R$  are  measurable in $t$ for all $(x,a)\in\R^n\times A$ and continuous in $(x,a)$ for all $t\in[0,T]$. Assume also the following:
\begin{enumerate}
\item[\te{\bf{(i)}}] for every $R>0$ there exists a measurable map  $k_R:[0,T]\to[0,+\infty)$  such that $|f(t,x,a)-f(t,y,a)|\leq k_R(t)\,|x-y|$ and $|l(t,x,a)-l(t,y,a)|\leq k_R(t)\,|x-y|$ for every  $t\in[0,T]$, $x,y\in\B_R$, $a\in A$\te{;}\vspace{1mm}
\item[\te{\bf{(ii)}}] there exists a measurable map $c:[0,T]\to[0,+\infty)$ such that for every $t\in[0,T]$, $x\in\R^n$, $a\in A$ one has  $|f(t,x,a)|\leq c(t)(1+|x|)$.
\end{enumerate}
If the triple $(A,f,l)$ is a representation of $H$, then $H$ satisfies \te{(H1)$-$(H4)}, \te{(HLC)} and  $L(t,x,\cdot\,):=H^{\ast}(t,x,\cdot\,)$  satisfies  \te{(BLC)}. Moreover, if $\,f,l\,$ are continuous, then $\,H\,$ and $\,\lambda\,$ are also continuous.
\end{Cor}

\begin{Rem}
It follows from Theorem \ref{podr22_th_wk} that the condition (BLC) is necessary 
for the existence of a continuous and $t$-measurable faithful representation $(A,f,l)$
with a compact control set $A$. Therefore, neither continuous nor $t$-measurable faithful representations $(A,f,l)$ with the compact control set $A$ exists for the Hamiltonian from the Example \ref{ex-4} because the function $v\rightarrow L(t,x,v)$ from this example is not upper bounded on the effective domain.
\end{Rem}

\subsection{Sufficient condition for the existence of a faithful representation} This subsection is devoted to a new representation theorem  with the compact control set.

\begin{Th}[\bf{Representation}]\label{th-rprez-glo12}
Assume \te{(H1)$-$(H4)}, \te{(HLC)} and \te{(BLC)}. Then there exist  $f:[0,T]\times\R^n\times \B\rightarrow\R^n$ and $l:[0,T]\times\R^n\times \B\rightarrow\R$, measurable in $t$ for all $(x,a)\in\R^n\times \B$ and continuous in $(x,a)$ for all $t\in[0,T]$, such that for every $t\in[0,T]$, $x,p\in\R^n$
\begin{equation*}
 H(t,x,p)=\sup_{\;\;a\in \B}\,\{\,\langle\, p,f(t,x,a)\,\rangle-l(t,x,a)\,\}
\end{equation*}
and $f(t,x,\B)=\D H^{\ast}(t,x,\cdot)$, where $\B$ is the closed unit ball in $\R^{n+1}$\!. 
Moreover, we have:
\begin{enumerate}
\item[\te{\bf{(A1)}}] For any $R>0$ and for all $t\in[0,T]$, $x,y\in \B_R$, $a,b\in\B$\vspace{-1mm}
\begin{equation*}
\begin{array}{l}
|f(t,x,a)-f(t,y,b)|\leq 10(n+1)(\omega_R(t)+3(1+R)\,k_R(t)+1)(|x-y|+|a-b|),\\[0.3mm]
|l(t,x,a)-l(t,y,b)|\leq 10(n+1)(\omega_R(t)+3(1+R)\,k_R(t)+1)(|x-y|+|a-b|),\\[0.3mm]
\it{where}\; \omega_R(t):=|\lambda(t,0)|+|H(t,0,0)|+c(t)(2+R).
\end{array}\vspace{-1mm}
\end{equation*}
\item[\te{\bf{(A2)}}] $|f(t,x,a)|\leq c(t)(1+|x|)$ for all $t\in[0,T]$, $x\in\R^n$, $a\in\B$.\vspace{1mm}
\item[\te{\bf{(A3)}}] $\G H^{\ast}(t,x,\cdot)\subset\bigcup_{a\in\B}(f(t,x,a),l(t,x,a))$ for all $t\in[0,T]$, $x\in\R^n$.\vspace{1mm}
\item[\te{\bf{(A4)}}] Furthermore, if $H$, $\lambda(\cdot,\cdot)$, $c(\cdot)$ are continuous, so are $f,l$.
\end{enumerate}
\end{Th}

The proof of Theorem \ref{th-rprez-glo12} is given in Section \ref{pofrepth}.
Now we point out the differences between our construction of a faithful representation and the ones presented in \cite{F-S,FR}. In order to do this, we consider two following examples.

\begin{Ex}\label{prz-rep-111}
Let the Hamiltonian $H$ be as in Example \ref{ex-1}. This Hamiltonian satisfies
assumptions (H1)$-$(H4),\,(HLC) and (BLC).
Our construction of representation $(A,f,l)$ of this Hamiltonian
leads to the set $A=[-1,1]\times[-1,1]$ and functions:
\begin{equation*}
f(x,a_1,a_2)=a_1|x|, \qquad l(x,a_1,a_2)=|a_1|+|a_2|(1-|a_1|),
\end{equation*}
that are the Lipschitz continuous. However, construction of representation  $(A,f,l)$ of this Hamiltonian that is  presented in \cite{F-S,FR} leads to the set  $A=[-1,1]$ and functions:
\begin{equation*}
f(x,a)=a|x|, \qquad
l(x,a)=L(x,f(x,a))=\left\{
\begin{array}{ccl}
|a| & \te{if} & x\not=0 \\[1mm]
0 & \te{if} & x=0.
\end{array}
\right.
\end{equation*}
We notice that the function  $l$ is discontinuous with respect to $x$ for all $a\in[-1,1]\setminus\{0\}$.
\end{Ex}

\begin{Ex}\label{prz-rep-112}
Let the Hamiltonian $H$ be as in Example \ref{ex-2}. This Hamiltonian satisfies
assumptions (H1)$-$(H4),\,(HLC) and (BLC).
Our construction of representation $(A,f,l)$ of this Hamiltonian
leads to the set $A=\{(a_1,a_2)\in\R\times\R\mid a_1^2+a_2^2\leq 1\}$ and functions:
\begin{equation*}
f(x,a_1,a_2)=a_1, \qquad l(x,a_1,a_2)=a_2+|x|,
\end{equation*}
that satisfy the Lipschitz continuity.  However, construction of representation  $(A,f,l)$ of this Hamiltonian that is presented in \cite{F-S,FR} leads to the set $A=[-1,1]$ and functions:
\begin{equation*}
f(x,a)=a, \qquad
l(x,a)=L(x,f(x,a))=|x|-\sqrt{1-a^2}.
\end{equation*}
We notice that the function $l$ is continuous, but not Lipschitz continuous with respect to the variable  $a$. 
\end{Ex}

\begin{Rem}
If (H1)$-$(H4), (HLC), (BLC) hold for a.e. $t\in[0,T]$, then  the conclusion of Theorem \ref{th-rprez-glo12} also holds for a.e. $t\in[0,T]$. Indeed, to show this, we simply redefine the Hamiltonian. More precisely, if $H$ satisfies (H1)$-$(H4), (HLC), (BLC) for a.e. $t\in[0,T]$, then there exist  a measure zero set $\mathcal{N}$ and a Hamiltonian $\tilde{H}$ such that $\tilde{H}(t,\cdot,\cdot)=0$ for all $t\in\mathcal{N}$ and $\tilde{H}(t,\cdot,\cdot)=H(t,\cdot,\cdot)$ for all $t\in[0,T]\setminus\mathcal{N}$. Moreover, $\tilde{H}$ satisfies  (H1)$-$(H4), (HLC), (BLC) for all $t\in[0,T]$.
\end{Rem}

\vspace*{-4mm}
\pagebreak

\subsection{Stability of representations}\label{subsecstarep}
In this subsection we will see that the faithful representation obtained in the previous subsection is stable.

\begin{Th}\label{thm-rep-stab2}
Let $H_i,H:[0,T]\times\R^n\times\R^n\to\R$, $i\in\N$, satisfy \te{(H1)$-$(H4)}, \te{(HLC)}. Assume that $L_i$, $L$, $i\in\N$, are given by \eqref{tran1} and satisfy \te{(BLC)}. Let $H_i,\lambda_i,c_i$, $i\in\N$, be continuous functions. We consider the representations $(\B,f_i,l_i)$ and $(\B,f,l)$ of $H_i$ and $H$, respectively,  defined as in the proof of Theorem~\ref{th-rprez-glo12}. If $H_i,\lambda_i,c_i$ converge uniformly on compacts to $H,\lambda,c$, respectively, then $f_i$ converge to $f$ and $l_i$ converge to $l$ uniformly on compacts in $[0,T]\times\R^n\times\B$.
\end{Th}

\begin{Th}\label{thm-rep-stab4}
Let $H_i,H:[0,T]\times\R^n\times\R^n\to\R$, $i\in\N$, satisfy \te{(H1)$-$(H4)}, \te{(HLC)}. Assume that $L_i$, $L$, $i\in\N$, are given by \eqref{tran1} and satisfy \te{(BLC)}. We consider the representations $(\B,f_i,l_i)$ and $(\B,f,l)$ of $H_i$ and $H$, respectively, defined as in the proof of Theorem~\ref{th-rprez-glo12}. If $H_i(t,\cdot,\cdot),\;\lambda_i(t,\cdot)$ converge uniformly on compacts to $H(t,\cdot,\cdot),\;\lambda(t,\cdot)$ , respectively, and $c_i(t)\to c(t)$ for all $t\in[0,T]$, then $f_i(t,\cdot,\cdot)$ converge to $f(t,\cdot,\cdot)$ and $l_i(t,\cdot,\cdot)$ converge to $l(t,\cdot,\cdot)$ uniformly on compacts in $\R^n\times\B$ for all $t\in[0,T]$.
\end{Th}

The proofs of Theorems  \ref{thm-rep-stab2} and \ref{thm-rep-stab4} are given in Section \ref{thms-stab}.

The following corollary is a consequence of Theorem \ref{thm-rep-stab2} and Gronwall's Lemma.

\begin{Cor}\label{cor-rep-stab2}
Let $H_i,H:[0,T]\!\times\!\R^n\!\times\!\R^n\to\R$, $i\in\N$, satisfy \te{(H1)$-$(H4)}, \te{(HLC)}. Assume that $L_i$, $L$, $i\in\N$, are given by \eqref{tran1} and satisfy \te{(BLC)}. Let $H_i,\lambda_i,k_{R\,i},c_i,g_i$, $i\in\N$, be continuous functions and converge uniformly on compacts to $H,\lambda,k_R,c,g$, respectively.  We consider the representations $(\B,f_i,l_i)$ and $(\B,f,l)$ of $H_i$ and $H$, respectively,  defined as in the proof of Theorem~\ref{th-rprez-glo12}. If $V_i$ and $V$ are the value functions associated with $(\B,f_i,l_i,g_i)$ and $(\B,f,l,g)$, respectively, then $V_i$ converge uniformly on compacts to $V$ in $[0,T]\times\R^n$.
\end{Cor}

\begin{Def}
A sequence of functions $\{\varphi_i\}_{i\in\N}$, is said to \it{epi-converge} to  function $\varphi$ (e-$\lim_{i\to\infty}\varphi_i=\varphi$ for short) if, for every point $x\in\R^n$,
\begin{enumerate}
\item[\bf{(i)}] $\liminf_{i\to\infty}\varphi_i(x_i)\geq\varphi(x)$ for every sequence $x_i\to x$,
\item[\bf{(ii)}] $\limsup_{i\to\infty}\varphi_i(x_i)\leq\varphi(x)$ for some sequence $x_i\to x$.
\end{enumerate}
\end{Def}

The following corollary is a consequence of Theorem \ref{thm-rep-stab4} and Gronwall's Lemma.

\begin{Cor}\label{cor-rep-stab4}
Let $H_i,H:[0,T]\!\times\!\R^n\!\times\!\R^n\to\R$, $i\in\N$, satisfy \te{(H1)$-$(H4)}, \te{(HLC)}  with the same  integrable functions $c(\cdot)$, $k_R(\cdot)$. Assume that $L_i$, $L$, $i\in\N$, are given by \eqref{tran1} and satisfy \te{(BLC)} with the same  integrable function $k_R(\cdot)$. Let $g_i$, $g$, $i\in\N$, be proper, lower semicontinuous and \tn{e-$\lim_{i\to\infty}g_i=g$}.  Assume that there exists an integrable function $\mu(\cdot)$ such that $|H_i(t,0,0)|\leq\mu(t)$ and $|\lambda_i(t,0)|\leq\mu(t)$ for all $t\in[0,T]$, $i\in\N$.   We consider the representations $(\B,f_i,l_i)$ and $(\B,f,l)$ of $H_i$ and $H$, respectively, defined as in the proof of Theorem~\ref{th-rprez-glo12}. Assume that $V_i$ and $V$ are the value functions associated with $(\B,f_i,l_i,g_i)$ and $(\B,f,l,g)$, respectively. 
If $H_i(t,\cdot,\cdot)$ converge to $H(t,\cdot,\cdot)$ and $\lambda_i(t,\cdot)$ converge to $\lambda(t,\cdot)$ uniformly on compacts for all $t\in[0,T]$, then \tn{e-$\lim_{i\to\infty}V_i=V$}.
\end{Cor}

\subsection{Reduction a variational problem to an optimal control problem}\label{rvptoocp}
The indicator function $\psi_S(\cdot)$  of a set $S$ is given by $0$ on this set but $+\infty$ outside.
Let $\mathcal{A}([0,1],\R^n)$ be the space of all absolutely continuous functions.

 We consider the following generalized variational problem:
\begin{equation}
\begin{aligned}
\mathrm{minimize}&\;\;\;\Gamma[x(\cdot)]:=\phi(x(0),x(1))+\int_{0}^1L(t,x(t),\dot{x}(t))\,dt,\\[-1mm]
\mathrm{subject\;\, to}&\;\;\;x(\cdot)\in \mathcal{A}([0,1],\R^n).
\end{aligned}\tag{$\mathcal{P}_{v}$}
\end{equation}

We consider the following  optimal control problem: 
\begin{equation}
\begin{aligned}
\mathrm{minimize}&\;\;\;\Lambda[(x,a)(\cdot)]:=\phi(x(0),x(1))+\int_{0}^1l(t,x(t),a(t))\,dt,\\[-0.5mm]
\mathrm{subject\;\, to}&\;\;\;\dot{x}(t)=f(t,x(t),a(t)),\;\; a(t)\in\B\;\;\;\mathrm{a.e.}
\;\;t\in[0,1],\\
\mathrm{and}&\;\;\; x(\cdot)\in \mathcal{A}([0,1],\R^n),\;a(\cdot)\in L^1([0,1],\R^{n+1}).
\end{aligned}\tag{$\mathcal{P}_{c}$}
\end{equation}

\begin{Th}\label{thm-reduct}
Assume that \te{(H1)$-$(H4)}, \te{(HLC)}, \te{(BLC)} hold with integrable functions $c(\cdot)$, $k_R(\cdot)$, $H(\cdot,0,0)$, $\lambda(\cdot,0)$.  We consider the representation $(\B,f,l)$ of $H$ defined as in Theorem~\ref{th-rprez-glo12}. Assume further that $\phi$ is a proper, lower semicontinuous function  and  there exists $M\geq 0$ such that $\min\{\,|z|,|x|\,\}\leq M$ for all $(z,x)\in\D\phi$. Then
\begin{equation*}
\min\Gamma[x(\cdot)]\,=\,\min\Lambda[(x,a)(\cdot)].
\end{equation*}
Besides, if $\kr{x}(\cdot)$ is the optimal arc of $(\mathcal{P}_{v})$  such that $\kr{x}(\cdot)\in\D\Gamma$, then there exists  $\kr{a}(\cdot)$ such that $(\kr{x},\kr{a})(\cdot)$ is the optimal arc of $(\mathcal{P}_{c})$ and $(\kr{x},\kr{a})(\cdot)\in\D\Lambda$.  Conversely, if $(\kr{x},\kr{a})(\cdot)$ is the optimal arc of $(\mathcal{P}_{c})$, then $\kr{x}(\cdot)$ is the optimal arc of $(\mathcal{P}_{v})$.
\end{Th}

The proof of Theorem  \ref{thm-reduct} is given in Section \ref{thm-reduct-sect}.

  Applying Theorem \ref{thm-reduct} to $\phi(z,x):=\psi_{\{x_0\}}(z)+g(x)$, we obtain the following corollary:

\begin{Cor}\label{cor-reduct}
Assume that \te{(H1)$-$(H4)}, \te{(HLC)}, \te{(BLC)} hold with integrable functions $c(\cdot)$, $k_R(\cdot)$, $H(\cdot,0,0)$, $\lambda(\cdot,0)$.  We consider the representation $(\B,f,l)$ of $H$ defined as in  Theorem~\ref{th-rprez-glo12}. Assume further that $g$ is a proper, lower semicontinuous function. If $V$ is the value function associated with $(\B,f,l,g)$, then for all $(t_0,x_0)\in[0,T]\times\R^n$
\begin{eqnarray*}
V(t_0,x_0) &=& \min_{\begin{array}{c}
\scriptstyle x(\cdot)\,\in\,\mathcal{A}([t_0,T],\R^n)\\[-1mm]
\scriptstyle x(t_0)=x_0
\end{array}}\!\!\big\{\,g(x(T))+\int_{t_0}^TL(t,x(t),\dot{x}(t))\,dt\,\big\}\\
&=& \min_{(x,a)(\cdot)\,\in\, \emph{S}_f(t_0,x_0)}\,\big\{\,g(x(T))+\int_{t_0}^Tl(t,x(t),a(t))\,dt\,\big\}.
\end{eqnarray*}
\end{Cor}

\begin{Rem}
Using  Corollary \ref{cor-reduct} we can prove that if  $g$ is locally Lipschitz  continuous/continuous/lower semicontinuous, so is $V$.
\end{Rem}


\section{Proof of Theorem \ref{podr22_th_wk}}\label{wk-kon-istr}

\noindent At the beginning we prove three lemmas which will be used in the proof of
Theorem \ref{podr22_th_wk}.

\begin{Lem}\label{lem-tran-r1}
Assume that  $p\rightarrow H(t,x,p)$ is a real-valued convex function. If the triple $(A,f,l)$ is a representation of $H$ with a nonempty set $A$, then $L(t,x,\cdot\,):=H^{\ast}(t,x,\cdot\,)$ satisfies the inequality $L(t,x,f(t,x,a))\leq l(t,x,a)$ for all $a\in A$.
\end{Lem}

\vspace*{-4mm}
\pagebreak

\begin{proof}
We assume, by contradiction, that the claim is false. Then there exists $\kr{a}\in A$ such that  $l(t,x,\kr{a})<L(t,x,f(t,x,\kr{a}))$. Therefore $(f(t,x,\kr{a}), l(t,x,\kr{a}))\not\in\E L(t,x,\cdot)$. Because $p\rightarrow H(t,x,p)$ is finite and convex, the function $v\rightarrow L(t,x,v)$ is proper, convex, lower semicontinuous and  $H(t,x,\cdot\,)=L^{\ast}(t,x,\cdot\,)$. Hence the set $\E L(t,x,\cdot)$ is nonempty, closed and convex. By Epigraph Separation Theorem, there exists $\kr{p}\in\R^n$ such that
\begin{equation}\label{stod-dlae}
\sup_{(v,\eta)\in\E L(t,x,\cdot)} \langle\, (v,\eta),(\kr{p},-1)\,\rangle \;<\; \langle\, (f(t,x,\kr{a}), l(t,x,\kr{a})),(\kr{p},-1)\,\rangle.
\end{equation}
We note that $(v,L(t,x,v))\in\E L(t,x,\cdot)$ for all $v\in\D L(t,x,\cdot)$. So, by the inequality
 (\ref{stod-dlae}) and the equality $H(t,x,\cdot\,)=L^{\ast}(t,x,\cdot\,)$ we obtain
\begin{eqnarray*}
H(t,x,\kr{p}) &=&\sup_{v\,\in\,\D L(t,x,\cdot)}\langle (v,L(t,x,v)),(\kr{p},-1)\rangle\\ 
&\leq & \sup_{(v,\eta)\,\in\,\E L(t,x,\cdot)} \langle\, (v,\eta),(\kr{p},-1)\,\rangle\\
&<& \langle\, (f(t,x,\kr{a}), l(t,x,\kr{a})),(\kr{p},-1)\,\rangle\\
&\leq & H(t,x,\kr{p}).
\end{eqnarray*}
Thus, we have a contradiction, that completes the proof.
\end{proof}

\begin{Lem}\label{lem-tran-r2}
Assume that the set $A$ is nonempty and compact. Let $a\to f(t,x,a)$ and $a\to l(t,x,a)$ be continuous functions and the set $f(t,x,A)$ be convex. If the triple $(A,f,l)$ is a representation of $H$, then
$f(t,x,A)=\D L(t,x,\cdot)$, where $L(t,x,\cdot\,):=H^{\ast}(t,x,\cdot\,)$.
\end{Lem}

\begin{proof}
Because $p\rightarrow H(t,x,p)$ is finite and convex, the function $v\rightarrow L(t,x,v)$ is proper, convex and lower semicontinuous. By Lemma \ref{lem-tran-r1} we have  $L(t,x,f(t,x,a))\leq l(t,x,a)$ for every  $a\in A$. Hence we obtain $f(t,x,A)\subset \D L(t,x,\cdot)$. Now we show that $\D L(t,x,\cdot)\subset f(t,x,A)$. We suppose that this inclusion is false.
Then there exists an element $\kr{v}\in \D L(t,x,\cdot)$ and $\kr{v}\not\in f(t,x,A)$. The set $f(t,x,A)$ is nonempty, convex and compact, so by the Separation Theorem,
there exist an element $\kr{p}\in\R^n$ and numbers $\alpha,\beta\in\R$ such that
\begin{equation*}
 \langle\,\kr{v},\kr{p} \,\rangle\leq \alpha<\beta\leq \langle\, f(t,x,a),\kr{p}\,\rangle,  \;\;\;\forall\;a\in A.
\end{equation*}
We notice that by the above inequality we obtain 
\begin{equation}\label{stod-dlae0101}
 \beta-\alpha\;\leq\; \langle\, f(t,x,a)-\kr{v},\kr{p}\,\rangle,  \;\;\;\forall\;a\in A.
\end{equation}
Put $\xi(t,x):=\inf_{a\in A}l(t,x,a)$. Let $\kr{n}\in N$ be large enough for the following inequality to hold
\begin{equation}\label{stod-dlae01011}
L(t,x,\kr{v})-\xi(t,x)\;<\;\kr{n}\cdot(\beta-\alpha).
\end{equation}
 Our assumptions imply that for  $\,\kr{q}:=-\kr{n}\cdot\kr{p}\,$  there exists  $\,a_{\kr{q}}\in A\,$  such that
\begin{equation}\label{stod-dlae01012}
H(t,x,\kr{q})=\langle\, \kr{q},f(t,x,a_{\kr{q}})\,\rangle-l(t,x,a_{\kr{q}}).
\end{equation} 
From  (\ref{stod-dlae01011}), (\ref{stod-dlae01012}) and (\ref{stod-dlae0101}), it follows that
\begin{eqnarray*}
\kr{n}\cdot(\beta-\alpha) 
&>& \langle\,\kr{v},\kr{q}\,\rangle - H(t,x,\kr{q})-\xi(t,x)\\
&\geq & \langle\, \kr{v}-f(t,x,a_{\kr{q}}),\kr{q}\,\rangle\\
&\geq & \kr{n}\cdot(\beta-\alpha).
\end{eqnarray*}
Thus, we obtain a contradiction, that completes the proof.
\end{proof}

\begin{Lem}\label{lem-tran-r1w1}
Assume that the set $A$ is nonempty and compact. Let  $f:[0,T]\times\R^n\times A\rightarrow\R^n$ and $l:[0,T]\times\R^n\times A\rightarrow\R$ be measurable in $t$ for all $(x,a)\in\R^n\times A$ and continuous in $(x,a)$ for all $t\in[0,T]$. If the triple $(A,f,l)$ is a representation of $H$,
then there exist a nonempty, compact set $\mathbbmtt{A}$ and functions $\mathbbmtt{f}:[0,T]\times\R^n\times \mathbbmtt{A}\rightarrow\R^n$ and $\mathbbmtt{l}:[0,T]\times\R^n\times \mathbbmtt{A}\rightarrow\R$, measurable in $t$ for all $(x,\mathbbmtt{a})\in\R^n\times \mathbbmtt{A}$ and continuous in $(x,\mathbbmtt{a})$ for all $t\in[0,T]$, such that the triple $(\mathbbmtt{A},\mathbbmtt{f},\mathbbmtt{l})$ is a representation of $H$. Moreover, for all $t\in[0,T]$, $x\in\R^n$
\begin{eqnarray}\label{rwr-lem0}
\mathbbmtt{f}(t,x,\mathbbmtt{A})=\te{conv}f(t,x,A),\qquad \mathbbmtt{l}(t,x,\mathbbmtt{A})=\te{conv}\,l(t,x,A).
\end{eqnarray}

Furthermore, if for any $R>0$ there exists a measurable map  $k_R:[0,T]\to[0,+\infty)$ such that $l(t,\cdot,a)$ is $k_R(t)$-Lipschitz on $\B_R$ for every  $t\in[0,T]$ and $a\in A$, then $\mathbbmtt{l}(t,\cdot,\mathbbmtt{a})$ is also $k_R(t)$-Lipschitz on $\B_R$ for every  $t\in[0,T]$ and $\mathbbmtt{a}\in \mathbbmtt{A}$.

Besides, if functions  $f,l$ are continuous, then functions $\mathbbmtt{f},\mathbbmtt{l}$
are also continuous.
\end{Lem}

\begin{proof}
We define a simplex in the space $\R^{n+1}$ by
$$\Delta:= \{(\alpha_0,\dots,\alpha_{n})\in[0,1]^{n+1}\mid\alpha_0+\dots+\alpha_{n}=1\}.$$
Obviously, the set $\Delta$ is compact. Moreover, we define the set $\,\mathbbmtt{A}\,$
by $\mathbbmtt{A}:=A^{n+1}\times\Delta$. We notice that the set $\mathbbmtt{A}$
is compact. The functions $\mathbbmtt{f}$, $\mathbbmtt{l}$ are defined for every $t\in[0,T]$, $x\in\R^n$ and $\mathbbmtt{a}=(a_0,\dots,a_{n},\alpha_0,\dots,\alpha_{n})\in A^{n+1}\times\Delta=\mathbbmtt{A}$ by the formulas:
$$\mathbbmtt{f}(t,x,\mathbbmtt{a}):= \sum_{i=0}^{n}\alpha_i\, f(t,x,a_i),\qquad \mathbbmtt{l}(t,x,\mathbbmtt{a}):= \sum_{i=0}^{n}\alpha_i\, l(t,x,a_i).$$
We notice that $\mathbbmtt{f}$, $\mathbbmtt{l}$ are  measurable in $t$ for all $(x,\mathbbmtt{a})\in\R^n\times \mathbbmtt{A}$ and continuous in $(x,\mathbbmtt{a})$ for all $t\in[0,T]$. Besides, if functions $f,l$ are continuous, then functions $\mathbbmtt{f},\mathbbmtt{l}$ are also continuous.
It is not difficult to show that the triple  $(\mathbbmtt{A},\mathbbmtt{f},\mathbbmtt{l})$ is the representation of $H$ and $\mathbbmtt{l}(t,\cdot,\mathbbmtt{a})$ is $k_R(t)$-Lipschitz on $\B_R$ for every  $t\in[0,T]$ and $\mathbbmtt{a}\in \mathbbmtt{A}$.. 

Equalities  (\ref{rwr-lem0})  follow from the definition of the triple
$(\mathbbmtt{A},\mathbbmtt{f},\mathbbmtt{l})$ and Carath\'eodory's Theorem  (convex hull), cf. \cite[Thm. 2.29]{R-W}.
\end{proof}

\begin{proof}[Proof of Theorem \ref{podr22_th_wk}]
By Lemma  \ref{lem-tran-r1w1} there exist a nonempty, compact set $\mathbbmtt{A}$ and
functions $\mathbbmtt{f}$, $\mathbbmtt{l}$ measurable in $t$ for all $(x,\mathbbmtt{a})\in\R^n\times \mathbbmtt{A}$ and continuous in $(x,\mathbbmtt{a})$ for all $t\in[0,T]$ such that the triple $(\mathbbmtt{A},\mathbbmtt{f},\mathbbmtt{l})$ is a representation of $H$ and $\mathbbmtt{f}(t,x,\mathbbmtt{A})=\te{conv}f(t,x,A)$ for every $t\in[0,T]$, $x\in\R^n$. Therefore, by Lemma \ref{lem-tran-r2} we have for all $t\in[0,T]$, $x\in\R^n$
\begin{eqnarray}\label{podr221}
\mathbbmtt{f}(t,x,\mathbbmtt{A})=\D L(t,x,\cdot).
\end{eqnarray}

Now, we prove that the condition (BLC) holds. Let us put 
$$\lambda(t,x):=\sup_{\mathbbmtt{a}\in \mathbbmtt{A}}\mathbbmtt{l}(t,x,\mathbbmtt{a})$$
Obviously, the function $\lambda$ is measurable in $t$ for all $x\in\R^n$ and  continuous in $x$ for all $t\in[0,T]$. Let us fix $t\in[0,T]$ and $x\in\R^n$. If $\kr{v}\in\D L(t,x,\cdot)$, then by the equality (\ref{podr221}) there exists $\kr{\mathbbmtt{a}}\in \mathbbmtt{A}$ such that   $\kr{v}=\mathbbmtt{f}(t,x,\kr{\mathbbmtt{a}})$. Therefore by Lemma \ref{lem-tran-r1}
 $$L(t,x,\kr{v})=L(t,x,\mathbbmtt{f}(t,x,\kr{\mathbbmtt{a}}))\leq \mathbbmtt{l}(t,x,\kr{\mathbbmtt{a}})\leq \lambda(t,x).$$
It means that  $L(t,x,v)\leq \lambda(t,x)$ for every  $t\in[0,T]$, $x\in\R^n$, $v\in\D L(t,x,\cdot)$.

 By Lemma~\ref{lem-tran-r1w1} we have that  $\mathbbmtt{l}(t,\cdot,\mathbbmtt{a})$ is  $k_R(t)$-Lipschitz on $\B_R$ for any  $t\in[0,T]$, $\mathbbmtt{a}\in \mathbbmtt{A}$ and $R>0$. Let us fix $t\in[0,T]$, $x,y\in\B_R$ and $R>0$.  Let $\kr{\mathbbmtt{a}}\in \mathbbmtt{A}$ be such that  $\lambda(t,x)=\mathbbmtt{l}(t,x,\kr{\mathbbmtt{a}})$. Then we have
\begin{eqnarray*}
\lambda(t,x)-\lambda(t,y) &= & \mathbbmtt{l}(t,x,\kr{\mathbbmtt{a}})-\sup_{\mathbbmtt{a}\in \mathbbmtt{A}}\mathbbmtt{l}(t,y,\mathbbmtt{a})\\ & \leq &  \mathbbmtt{l}(t,x,\kr{\mathbbmtt{a}})-\mathbbmtt{l}(t,y,\kr{\mathbbmtt{a}})
\;\;\leq \;\; k_R(t)\,|x-y|.
\end{eqnarray*}
Since $t\in[0,T]$, $x,y\in\B_R$ and $R>0$ are arbitrary, so $\lambda(t,\cdot)$
is  $k_R(t)$-Lipschitz on $\B_R$ for any $t\in[0,T]$, $x,y\in\B_R$ and $R>0$.

Besides, if functions $f,l$ are continuous, then by Lemma \ref{lem-tran-r1w1}, the functions $\mathbbmtt{f},\mathbbmtt{l}$ are also continuous. Therefore, the function $\lambda$
has to be continuous.
\end{proof}


\section{Proof of representation theorem}\label{pofrepth}

\noindent In the beginning of this section  we  introduce some auxiliary definitions and lemmas. By $\mathcal{P}_{fc}(\R^m)$
we denote a family of all nonempty, closed and convex subsets of $\R^m$. Then, let  $\mathcal{P}_{kc}(\R^m)$ be a family of all nonempty, convex and compact subsets of $\R^m$.

\begin{Lem}[\te{\cite[p. 369]{A-F}}]\label{lem-pmh}
The set-valued map $P:\R^m\times\mathcal{P}_{fc}(\R^m)\multimap\mathcal{P}_{kc}(\R^m)$ defined by
\begin{equation*}
P(y,K):= K\cap\B(y,2d(y,K))
\end{equation*}
is Lipschitz with the Lipschitz constant $5$, i.e. for all $K,D\in\mathcal{P}_{fc}(\R^m)$ and $x,y\in\R^m$
\begin{equation*}
\mathscr{H}(P(x,K),P(y,D))\leq 5(\mathscr{H}(K,D)+|x-y|).
\end{equation*}
\end{Lem}

The support function $\sigma(K,\cdot):\R^m\to\R$ of the set $K\in\mathcal{P}_{kc}(\R^m)$ is a convex  real-valued function defined by
\begin{equation*}
\sigma(K,p):=\max_{x\in K}\,\langle p,x\rangle,\quad \forall\,p\in\R^m.
\end{equation*}

Let $\sum_{^{\, m-1}}$ denotes the unit sphere in $\R^m$ and let $\mu$ be the measure on $\sum_{^{\, m-1}}$ proportional to the Lebesgue measure and satisfying $\mu(\sum_{^{\, m-1}})=1$.

\begin{Def}\label{df-scel}
Let $m\in\N\setminus\{1\}$. For any $K\in\mathcal{P}_{kc}(\R^m)$, its Steiner point is defined by
\begin{equation*}
s_m(K):=m\int_{\sum_{^{\, m-1}}}p\,\sigma(K,p)\;\mu(dp).
\end{equation*}
One can show that $s_m(\cdot)$ is a selection in the sense that $s_m(K)\in K$, cf. \cite[p. 366]{A-F}.
\end{Def}

\begin{Lem}[\te{\cite[p. 366]{A-F}}]\label{lem-scmh}
The function $s_m(\cdot)$ is  Lipschitz in the Hausdorff metric with the Lipschitz constant $m$ on the set of all nonempty, convex and compact subsets of $\R^m$, i.e.
\begin{equation*}
|s_m(K)-s_m(D)|\leq m\, \mathscr{H}(K,D),\quad \forall\,K,D\in\mathcal{P}_{kc}(\R^m).
\end{equation*}
\end{Lem}

\begin{Lem}[\,\te{\cite[Chap. 5 and Chap. 14]{R-W}}\,]\label{dow-prp}
Let a set-valued map $E:[0,T]\times\R^n\multimap\R^m$ has nonempty, closed values. Assume that $E(\cdot,x)$ is measurable for every $x\in\R^n$ and $E(t,\cdot)$ has a closed
graph and is lower semicontinuous for every  $t\in[0,T]$. If a real-valued map $\omega(t,x)$, $(t,x)\in[0,T]\times\R^n$, is measurable in $t$ for every $x\in\R^n$ and continuous in $x$ for
every $t\in[0,T]$, then a real-valued map defined by
\begin{equation*}
 (t,x,a)\to d(\omega(t,x)\,a,E(t,x)),\quad \forall\,(t,x,a)\in[0,T]\times\R^n\times\R^m
\end{equation*}
is $t$-measurable for every $(x,a)\in\R^n\times\R^m$ and $(x,a)$-continuous for every  $t\in[0,T]$. In addition to this, it is a $(t,x,a)$-continuous map if $\omega$ is continuous, $E$ has
a closed graph and is lower semicontinuous.
\end{Lem}

\begin{Lem}[\,\te{\cite[Cor. 5.21]{R-W}}\,]\label{dow-prps}
Let a set-valued map $\Phi:[0,T]\times\R^k\multimap\R^m$ be locally bounded and has nonempty, compact values. Then $\Phi$ is continuous in the sense of the  Hausdorff's distance ($\mathscr{H}$-continuous) if and only if $\Phi$ has a closed graph and is lower semicontinuous.
\end{Lem}

The Hausdorff's distance between closed balls can be estimated in the following way:
\begin{equation}\label{dow-omk}
\mathscr{H}(\B(x,r),\B(y,s))\leq |x-y|+|r-s|,\;\; \forall x,y\in\R^n,\;\; \forall\,r,s\geq 0.
\end{equation}

\begin{Th}\label{th-oparam}
Let a set-valued map $E:[0,T]\times\R^n\multimap\R^m$ has nonempty, closed and
convex values. Assume that $E(\cdot,x)$ is measurable for every $x\in\R^n$ and $E(t,\cdot)$ has a closed graph and is lower semicontinuous for every $t\in[0,T]$. Let a real-valued map $\omega(t,x)\geq 1$, $(t,x)\in[0,T]\times\R^n$, be measurable in $t$ for all $x\in\R^n$ and continuous in $x$
for all $t\in[0,T]$. Then there exists a single-valued map $\e:[0,T]\times\R^n\times\R^m\to\R^m$ such that $\e(\cdot,x,a)$ is measurable for every  $(x,a)\in\R^n\times\R^m$ and $\e(t,\cdot,\cdot)$ is continuous  for every $t\in[0,T]$.

Moreover, for all $t\in[0,T]$, $x,y\in\R^n$, $a,b\in\R^m$ 
\begin{equation}\label{th-oparam-p}
[E(t,x)\cap \B_{\omega(t,x)}]\,\subset\,\e(t,x,\B)\,\subset\, E(t,x), 
\end{equation}
\begin{equation}\label{th-oparam-n}
|\e(t,x,a)-\e(t,y,b)|\leq 5m[\,\mathscr{H}(E(t,x),E(t,y))+|\omega(t,x)\,a-\omega(t,y)\,b|\,].
\end{equation}

Additionally, a single-valued map $\e$ is continuous if $\omega$ is continuous,  $E$ has a closed graph and is lower semicontinuous.
\end{Th}

\begin{proof}
Let $(t,x,a)\in[0,T]\times\R^n\times\R^m$. We consider the closed ball $G(t,x,a)\subset \R^m$ with the center $\omega(t,x)\,a$ and radius $2d(\omega(t,x)\,a,E(t,x))$, i.e.
$$G(t,x,a):=\B(\omega(t,x)\,a,2d(\omega(t,x)\,a,E(t,x))).$$
 By the inequality  \eqref{dow-omk}, Lemma \ref{dow-prp} and \cite[Cor. 8.2.13]{A-F} a set-valued map $G(\cdot,x,a)$ is measurable for every  $x\in\R^n$, $a\in\R^m$ and a  set-valued map $G(t,\cdot,\cdot)$ is $\mathscr{H}$-continuous for every $t\in[0,T]$.
Moreover, $\|G(t,x,a)\|\leq \varphi(t,x,a)$ for all $t\in[0,T]$, $x\in\R^n$, $a\in\R^m$, where
$$\varphi(t,x,a):= \omega(t,x)\,|a|+2d(\omega(t,x)\,a,E(t,x)).$$
By Lemma \ref{dow-prp} and our hypotheses, we obtain that $\varphi(\cdot,x,a)$
is measurable  for every $x\in\R^n$, $a\in\R^m$ and $\varphi(t,\cdot,\cdot)$ is continuous
for every $t\in[0,T]$.

Let $P$ be the map defined in Lemma \ref{lem-pmh}. We set
$$\Phi(t,x,a):= P(\omega(t,x)\,a,E(t,x))=E(t,x)\cap G(t,x,a).$$
By our hypotheses, the set $\Phi(t,x,a)$ is nonempty, compact
 and convex.  The maps $G(\cdot,x,a)$ and $E(\cdot,a)$ are measurable and have closed values,
so the map $\Phi(\cdot,x,a)$ which is their intersection is also measurable for all $x\in\R^n$, $a\in\R^m$, cf. \cite[Thm. 8.2.4]{A-F}.

Now we show that a  map $\Phi(t,\cdot,\cdot)$ is $\mathscr{H}$-continuous for every $t\in[0,T]$.  Because of Lemma \ref{dow-prps},  it is sufficient to show that for each fixed $t\in[0,T]$ the  map $\Phi(t,\cdot,\cdot)$ is locally bounded, has a closed graph and is lower semicontinuous.

The  map $\Phi(t,\cdot,\cdot)$ has a closed graph, because it is an intersection of maps $G(t,\cdot,\cdot)$ and $E(t,\cdot)$ which have closed graphs. Moreover, $\Phi(t,\cdot,\cdot)$ is locally bounded, because $\varphi(t,\cdot,\cdot)$ is continuous and $\|\Phi(t,x,a)\|\leq\|G(t,x,a)\|\leq \varphi(t,x,a)$ for every $(t,x,a)\in[0,T]\times\R^n\times\R^m$.

We prove that  $\Phi(t,\cdot,\cdot)$ is lower semicontinuous. Let us fix $(x,a)\in\R^n\times\R^m$ and the open set $O\subset\R^m$ such that $\Phi(t,x,a)\cap O\not=\emptyset$. We consider two cases.

Case 1. Let  $\I G(t,x,a)=\emptyset$. Then $G(t,x,a)\subset O$. We know that a map $G(t,\cdot,\cdot)$ has compact values and is $\mathscr{H}$-continuous. So  we have
$G(t,x',a')\subset O$ for all $(x',a')$ near $(x,a)$. Thus  $\Phi(t,x',a')\subset G(t,x',a')\subset O$ for all $(x',a')$ near $(x,a)$. Therefore  for every $(x',a')$ sufficiently close to $(x,a)$ we have $\Phi(t,x',a')\cap O\not=\emptyset$. It means that in this case, $\Phi(t,\cdot,\cdot)$ is lower semicontinuous.

Case 2. Let $\I G(t,x,a)\not=\emptyset$.  Then by the definition of $G(t,\cdot,\cdot)$ we deduce that there exists $z_2\in E(t,x)\cap \I G(t,x,a)$. We assume that $z_1\in \Phi(t,x,a)\cap O$. Then the interval $(z_1,z_2]\subset E(t,x)\cap \I G(t,x,a)$. Consequently, we can find an element $z\in\R^m$ satisfying $z\in O\cap E(t,x)\cap \I G(t,x,a)$. Hence, for some $\varepsilon>0$ we have $\B(z,\varepsilon)\subset G(t,x,a)\cap O$. The set-valued map $G(t,\cdot, \cdot)$ is a ball whose center and radius are continuous functions. Hence, for every $(x',a')$ sufficiently close to $(x,a)$ we have $\B(z,\varepsilon/2)\subset G(t,x',a')$. On the other hand, $E$ is lower semicontinuous, so $\B(z,\varepsilon/2)\cap E(t,x')\not=\emptyset$ for all $x'$ near $x$. Therefore for every $(x',a')$ sufficiently close to $(x,a)$ we have $\Phi(t,x',a')\cap O\not=\emptyset$. It means that also in this case, $\Phi(t,\cdot,\cdot)$ is lower semicontinuous.

We define the single-valued map $\e$ from $[0,T]\times\R^n\times\R^m$ to $\R^m$ by
$$\e(t,x,a):= s_m(\Phi(t,x,a)),$$
where $s_m$ in the Steiner selection. Since $\Phi$ is measurable with respect to $t$, using the  definition of $s_m$, we deduce that $\e$ is also measurable with respect to $t$. By Lemma \ref{lem-scmh} we have for all $t,s\in[0,T]$, $x,y\in\R^n$, $a,b\in\R^m$
\begin{equation}\label{dow-nreh}
|\e(t,x,a)-\e(s,y,b)|\leq m\mathscr{H}(\Phi(t,x,a),\Phi(s,y,b)).
\end{equation}
We have shown that  $\Phi(t,\cdot,\cdot)$ is $\mathscr{H}$-continuous for every $t\in[0,T]$. By the inequality \eqref{dow-nreh} we have that $\e(t,\cdot,\cdot)$ is continuous
for all $t\in[0,T]$. Additionally, if $E$ has a closed graph and is lower
semicontinuous, and  $\omega$ is continuous, then similarly to the above, one can prove
that    $\Phi$ is $\mathscr{H}$-continuous. Then by the inequality
 \eqref{dow-nreh} we have that a single-valued map  $\e$ is continuous.

We notice that by the inequality \eqref{dow-nreh} and Lemma \ref{lem-pmh}  for every $t\in[0,T]$, $x,y\in\R^n$, $a,b\in\R^m$ we obtain the inequality \eqref{th-oparam-n}.

Now we show that $[E(t,x)\cap \B_{\omega(t,x)}]\subset\e(t,x,\B)$ for every $(t,x)\in[0,T]\times\R^n$. For this purpose, we fix $(t,x)\in[0,T]\times\R^n$. Let $z\in E(t,x)\cap \B_{\omega(t,x)}$. Setting
$a:=z/\omega(t,x)$, we derive
$$a\in\B,\quad \omega(t,x)\,a=z,\quad \Phi(t,x,a)=\{z\}.$$
The above properties and Definition \ref{df-scel} imply that
$$z=s_m(\Phi(t,x,a))=\e(t,x,a)\in\e(t,x,\B).$$

We notice that by  Definition~\ref{df-scel} we obtain for all $t\in[0,T]$, $x\in\R^n$, $a\in\R^m$
$$\e(t,x,a)=s_m(\Phi(t,x,a))\in \Phi(t,x,a)\subset E(t,x).$$
This means that $\e(t,x,\B)\,\subset\, E(t,x)$ for every $(t,x)\in[0,T]\times\R^n$.
\end{proof}

\begin{Prop}\label{prop-reprezentacja H-ogr}
Let $A$ be a nonempty set and let $\e(t,x,\cdot)$ be a single-valued map defined on $A$ into $\R^n\times\R$.  Assume that $H(t,x,\cdot)$ is a real-valued convex function and 
\begin{equation}\label{par-epi}
\G L(t,x,\cdot)\subset \e(t,x,A)\subset \E L(t,x,\cdot),
\end{equation}
where $L(t,x,\cdot\,):=H^{\ast}(t,x,\cdot\,)$.
If $\e(t,x,a)=(f(t,x,a),l(t,x,a))$ for all $a\in A$, then the triple $(A,f,l)$ is a representation of $H$. Moreover, $f(t,x,A)=\D L(t,x,\cdot)$.
\end{Prop}

\begin{proof}
Because $H(t,x,\cdot)$ is finite and convex, the function $L(t,x,\cdot\,):=H^{\ast}(t,x,\cdot\,)$ is proper, convex, lower semicontinuous and  $H(t,x,\cdot\,)=L^{\ast}(t,x,\cdot\,)$.

Since $\e(t,x,a)\in \E L(t,x,\cdot)$ for every  $a\in A$, we have  $(f(t,x,a),l(t,x,a))\in \E L(t,x,\cdot)$ for all  $a\in A$. From the definition of $\E L(t,x,\cdot)$, it follows that $L(t,x, f(t,x,a))\leq l(t,x,a)$ for all $a\in A$. Hence $f(t,x,a)\in\D L(t,x,\cdot)$ for all  $a\in A$. Thus, for all  $a\in A$ and $p\in\R^n$
\begin{eqnarray*}
\langle p,f(t,x,a)\rangle-l(t,x,a) &\leq & \langle\, p,f(t,x,a)\,\rangle-L(t,x,f(t,x,a))\\
&\leq &\;\;\; \sup_{v\in\D L(t,x,\cdot)}\,\{\,\langle p,v\rangle-L(t,x,v)\,\}\;\;=\;\;H(t,x,p).
\end{eqnarray*}
Thus $f(t,x,A)\subset\D L(t,x,\cdot)$ and for every $p\in\R^n$
\begin{equation}\label{nrp1-org}
\sup_{a\in A}\,\{\,\langle\, p,f(t,x,a)\,\rangle-l(t,x,a)\,\}\leq H(t,x,p).
\end{equation}

Let us fix  $\kr{v}\in\D L(t,x,\cdot)$. Then $(\kr{v},L(t,x,\kr{v}))\in \G L(t,x,\cdot)$. Because of \eqref{par-epi}, there exists $\kr{a}\in A$ such that $(\kr{v},L(t,x,\kr{v}))=\e(t,x,\kr{a})=(f(t,x,\kr{a}),l(t,x,\kr{a}))$. Hence $\kr{v}=f(t,x,\kr{a})$ and $L(t,x,\kr{v})=l(t,x,\kr{a})$. Moreover, for every $p\in\R^n$
\begin{eqnarray*}
\langle p,\kr{v}\rangle-L(t,x,\kr{v}) &= & \langle\, p,f(t,x,\kr{a})\,\rangle-l(t,x,\kr{a})\\
&\leq & \sup_{a\in A}\,\{\,\langle\, p,f(t,x,a)\,\rangle-l(t,x,a)\,\}.
\end{eqnarray*}
Thus $\D L(t,x,\cdot)\subset f(t,x,A)$ and for every $p\in\R^n$
\begin{eqnarray}\label{nrp2-org}
\nonumber H(t,x,p) &= & \sup_{v\in\D L(t,x,\cdot)}\,\{\,\langle p,v\rangle-L(t,x,v)\,\}\\
&\leq & \sup_{a\in A}\,\{\,\langle\, p,f(t,x,a)\,\rangle-l(t,x,a)\,\}.
\end{eqnarray}

Combining inequalities (\ref{nrp1-org}) and (\ref{nrp2-org}) we obtain that the triple $(A,f,l)$ is a representation of $H$. Additionally, we have that  $f(t,x,A)=\D L(t,x,\cdot)$.
\end{proof}

\begin{Th}\label{do-th-parame-lo1}
Assume that  $H$ satisfies  \te{(H1)$-$(H4)} and \te{(HLC)}. Let $L$ be given by \eqref{tran1} and satisfy \te{(BLC)}. Then there exists a single-valued map $\e:[0,T]\times\R^n\times\B\to\R^{n}\times\R$ such that $\e(\cdot,x,a)$ is measurable for every  $(x,a)\in\R^n\times\B$ and $\e(t,\cdot,\cdot)$ is continuous  for every $t\in[0,T]$. Moreover, for all $t\in[0,T]$, $x\in\R^n$
\begin{equation}\label{do-th-parame-lo1-r1}
\G L(t,x,\cdot)\subset \e(t,x,\B)\subset \E L(t,x,\cdot).
\end{equation}
Furthermore, for any $R>0$ and for all $t\in[0,T]$, $x,y\in \B_R$, $a,b\in \B$
\begin{equation}\label{do-th-parame-lo1-r2}
\hspace*{7mm}\begin{array}{l}
|\e(t,x,a)-\e(t,y,b)|\leq 10(n+1)[\,k_R(t)\,|x-y|+|\omega(t,x)\,a-\omega(t,y)\,b|\,],\\[0.3mm]
\it{where}\; \omega(t,x):=|\lambda(t,x)|+|H(t,x,0)|+c(t)(1+|x|)+1.
\end{array}
\end{equation}
Additionally, if $H$, $\lambda(\cdot,\cdot)$, $c(\cdot)$ are continuous, so is $\e$.
\end{Th}

\begin{proof}
Let $\omega(t,x):=|\lambda(t,x)|+|H(t,x,0)|+c(t)(1+|x|)+1$ and $E(t,x):=E_L(t,x)$ for every
$(t,x)\in[0,T]\times\R^n$. Because of  Corollaries \ref{wrow-wm} and \ref{hlc-cor-ner},  the functions $\omega$ and $E$ satisfy assumptions of Theorem \ref{th-oparam}. Therefore, there exists a  map $\e:[0,T]\times\R^{n}\times\R^{n+1}\to\R^{n+1}$ such that $\e(\cdot,x,a)$ is measurable for every  $(x,a)\in\R^n\times\B$ and $\e(t,\cdot,\cdot)$ is continuous  for every $t\in[0,T]$. Moreover, it satisfies \eqref{th-oparam-p} and \eqref{th-oparam-n}.

 By the inequality \eqref{th-oparam-n} and Corollary~\ref{hlc-cor-ner} for all $t\in[0,T]$, $x,y\in \B_R$, $a,b\in \B$ and $R>0$
\begin{eqnarray*}
|\e(t,x,a)-\e(t,y,b)| &\leq & 5(n+1)[\,\mathscr{H}(E_L(t,x),E_L(t,y))+|\omega(t,x)\,a-\omega(t,y)\,b|\,]\\
&\leq & 10(n+1)\,k_R(t)\,|x-y|+5(n+1)|\omega(t,x)a-\omega(t,y)b|.
\end{eqnarray*}
It means that the inequality  \eqref{do-th-parame-lo1-r2} is satisfied. Moreover, if
we assume that $H$, $\lambda(\cdot,\cdot)$, $c(\cdot)$ are continuous, then $\omega$ is  continuous and  $E$ has a closed graph and is lower semicontinuous. Therefore, because of Theorem \ref{th-oparam}, we obtain that the map $\e$ is continuous.

Now we show that \eqref{do-th-parame-lo1-r1} holds. Because of \eqref{th-oparam-p},  it is sufficient to show that for each fixed $(t,x)\in[0,T]\times\R^n$  the following inclusion holds: 
\begin{equation}\label{dct-1}
\G L(t,x,\cdot)\subset[E(t,x)\cap \B_{\omega(t,x)}].
\end{equation}
Because $H(t,x,\cdot)$ is finite and convex,  the function $L(t,x,\cdot\,):=H^{\ast}(t,x,\cdot\,)$ is proper,\linebreak convex and lower semicontinuous. Let $(v,\eta)\in\G L(t,x,\cdot)$. Then from the definition~of $\G L(t,x,\cdot)$, it follows that $\eta=L(t,x,v)$. Hence  $(v,\eta)\in E(t,x)$ and $v\in\D L(t,x,\cdot)$. 
By Corollary \ref{wrow-wm} we get $\|\D L(t,x,\cdot)\|\leq c(t)(1+|x|)$. Therefore, 
\begin{equation}\label{dct-2}
|v|\leq c(t)(1+|x|).
\end{equation}
Moreover, because of \eqref{tran1} and (BLC),  we have
\begin{equation}\label{dct-3}
-|H(t,x,0)|\leq L(t,x,v)=\eta=L(t,x,v)\leq|\lambda(t,x)|.
\end{equation}
Combining inequalities \eqref{dct-2} and \eqref{dct-3} we obtain 
$$|(v,\eta)|\leq |v|+|\eta|\leq c(t)(1+|x|)+|\lambda(t,x)|+|H(t,x,0)|\leq \omega(t,x).$$
Consequently, we get  $(v,\eta)\in [E(t,x)\cap \B_{\omega(t,x)}]$. That completes the proof of \eqref{dct-1}.
\end{proof}

\begin{Rem}\label{dct-rem}
Let $\e:[0,T]\times\R^n\times\B\rightarrow\R^{n+1}$ be the function from Theorem \ref{do-th-parame-lo1}. We define two functions $f:[0,T]\times\R^n\times \B\rightarrow\R^n$ and $l:[0,T]\times\R^n\times \B\rightarrow\R$ by formulas:
\begin{equation*}
f(t,x,a):=\pi_1(\e(t,x,a))\;\;\;\te{and}\;\;\;l(t,x,a):=\pi_2(\e(t,x,a)),
\end{equation*}
where  $\pi_1(v,\eta)=v$ and $\pi_2(v,\eta)=\eta$ for all $v\in\R^n$ and $\eta\in\R$. Then for all $t\in[0,T]$, $x\in\R^n$, $a\in \B$ the following equality holds
\begin{equation*}
\e(t,x,a)=(f(t,x,a),l(t,x,a)).
\end{equation*}
Therefore, for all $t\in[0,T]$, $x,y\in\R^n$, $a,b\in \B$ we obtain
\begin{eqnarray*}
|f(t,x,a)-f(t,y,b)| &\!\!\leq\!\! &  |\e(t,x,a)-\e(t,y,b)|,\\
|\:l(t,x,a)\:-\:l(t,y,b)\:| &\!\!\leq\!\! & |\e(t,x,a)-\e(t,y,b)|.
\end{eqnarray*}
From the above inequalities it follows that  the  properties of the function $\e$ are inherited by
functions $f$ and $l$.
\end{Rem}

\begin{Rem}
It is not difficult to show that Theorem \ref{th-rprez-glo12} follows from  Proposition \ref{prop-reprezentacja H-ogr}, Theorem \ref{do-th-parame-lo1}, Remark \ref{dct-rem} and Corollary \ref{wrow-wm}.
\end{Rem}

\vspace*{-4mm}
\pagebreak


\section{Proofs of stability theorems}\label{thms-stab}
\noindent We show here that the faithful representation obtained in this paper is stable. To do this,
we need a few auxiliary definitions and facts. 
\begin{Def}
For a sequence $\{K_i\}_{i\in\N}$ of subsets of $\R^m$, the \it{upper limit} is the set
\begin{equation*}
\limsup_{i\to\infty}K_i:=\{\,x\in\R^m\;\mid\;\liminf_{i\to\infty}d(x,K_i)=0 \,\},
\end{equation*}
while the  \it{lower limit} is the set
\begin{equation*}
\liminf_{i\to\infty}K_i:=\{\,x\in\R^m\;\mid\;\limsup_{i\to\infty}d(x,K_i)=0 \,\}.
\end{equation*}
The \it{limit} of a sequence exists if the upper and lower limit sets are equal:
\begin{equation*}
\lim_{i\to\infty}K_i:= \limsup_{i\to\infty}K_i=\liminf_{i\to\infty}K_i.
\end{equation*}
\end{Def}

\begin{Rem} For nonempty, closed subsets $K_i$ and $K$ of $\R^m$, one has 
\begin{equation*}
\lim_{i\to\infty}K_i=K  \;\;\te{if and only if}\;\; \lim_{i\to\infty}d(x,K_i)=d(x,K)\;\;\te{for all}\;\; x\in\R^m,
\end{equation*}
cf. \cite[Cor. 4.7]{R-W}.
Thus, by the inequality $|d(x,K)-d(y,K)|\leq|x-y|$, that is satisfied for every
$x,y\in\R^m$ and every nonempty set $K\subset\R^m$, we obtain
 \begin{equation}\label{zmkno}
\lim_{i\to\infty}x_i=x_0,\quad\lim_{i\to\infty}K_i=K\quad\Longrightarrow\quad\lim_{i\to\infty}d(x_i,K_i)=d(x_0,K).
\end{equation}
\end{Rem}
\begin{Lem}[\te{\cite[Chap. 4, Sec C.]{R-W}}]\label{zmkh}
If $K_i$ and $K$ are nonempty, closed subsets of a given compact set in $\R^m$, then we have 
\begin{equation*}
\lim_{i\to\infty}K_i=K\;\iff\; \lim_{i\to\infty}\mathscr{H}(K_i,K)=0.
\end{equation*}
\end{Lem}

\begin{Lem}[\te{\cite[Thm. 4.32]{R-W}}]\label{thmozkp}
Let  $K_i$ and $D_i$ be convex sets in $\R^m$ for all $i\in\N$. If convex sets $K$ and $D$ satisfy $K\cap \I D\neq\emptyset$, then the following implication holds: $$\lim_{i\to\infty}K_i=K,\quad \lim_{i\to\infty}D_i=D\quad\Longrightarrow\quad \lim_{i\to\infty}\left(\,K_i\cap D_i\,\right)=K\cap D.$$
\end{Lem}

The following lemma is a consequence of Wijsman's Theorem, cf. \cite[Thm. 11.34]{R-W},

\begin{Lem}\label{klconv}
Assume that  $H_i:[0,T]\times\R^n\times\R^n\to\R$, $i\in\N\cup\{0\}$, are continuous and satisfy \tn{(H3)}. Let $L_i$, $i\in\N\cup\{0\}$, be given by \eqref{tran1}.  If $H_i$ converge to $H_0$ uniformly on compacts in $[0,T]\times\R^n\times\R^n$, then for every $(t_0,x_0)\in[0,T]\times\R^n$ we have
\begin{equation}\label{thm-zjnzz3}
\lim_{i\to\infty}E_{L_i}(t_i,x_i)=E_{L_0}(t_0,x_0)\;\;\it{for every sequence}\;\;(t_i,x_i)\to(t_0,x_0).
\end{equation}
\end{Lem}

\subsection{Proofs of stability theorems.}
Assume that  $H_i:[0,T]\times\R^n\times\R^n\to\R$, $i\in\N\cup\{0\}$, are continuous and satisfy (H3).   Let $L_i$, $i\in\N\cup\{0\}$, be given by  \eqref{tran1}. We consider continuous real-valued maps $\omega_i(t,x)\geq 1$, $(t,x)\in[0,T]\times\R^n$, $i\in\N\cup\{0\}$.

Let $(t,x,a)\in[0,T]\times\R^n\times\R^{n+1}$ and $i\in\N\cup\{0\}$. We consider the closed balls
\begin{equation*}
G_i(t,x,a):=\B(\omega_i(t,x)\,a,2d(\omega_i(t,x)\,a,E_{L_i}(t,x))).
\end{equation*}

\noindent We notice that $\|G_i(t,x,a)\|\leq \varphi_i(t,x,a)$, where
\begin{equation*}\varphi_i(t,x,a):= \omega_i(t,x)\,|a|+2d(\omega_i(t,x)\,a,E_{L_i}(t,x)).
\end{equation*}

\noindent Let $P$ be the map defined in Lemma \ref{lem-pmh}. We define the following sets
\begin{equation*}
\Phi_i(t,x,a):= P(\omega_i(t,x)\,a,E_{L_i}(t,x))=E_{L_i}(t,x)\cap G_i(t,x,a)
\end{equation*}
By Corollary \ref{wrow-wm} and our hypotheses, the sets $\Phi_i(t,x,a)$ are
nonempty, compact, convex.

We define the single-valued maps $\e_i$, $i\in\N\cup\{0\}$, from $[0,T]\times\R^n\times\R^{n+1}$ to $\R^{n+1}$ by
\begin{equation}\label{stbprof0}
    \e_i(t,x,a):= s_{n+1}(\Phi_i(t,x,a))
\end{equation}
where $s_{n+1}$ is the Steiner selection. By Lemma \ref{lem-scmh} we have
\begin{equation}\label{stbprof1}
|\e_i(t,x,a)-\e_0(s,y,b)|\leq (n+1)\mathscr{H}(\Phi_i(t,x,a),\Phi_0(s,y,b))
\end{equation}
for all $t,s\in[0,T]$, $x,y\in\R^n$, $a,b\in\R^{n+1}$, $i\in\N$.

\begin{Th}\label{thm-zjnzz}
Let $H_i$, $L_i$, $\omega_i$, $\e_i$,  $i\in\N\cup\{0\}$ be as above. If $H_i$ converge to $H_0$ uniformly on compacts in $[0,T]\times\R^n\times\R^n$ and $\omega_i$ converge to $\omega_0$ uniformly on compacts in $[0,T]\times\R^n$, then for every $(t_0,x_0,a_0)\in[0,T]\times\R^n\times\R^{n+1}$ we have
\begin{equation*}
\e_i(t_i,x_i,a_i)\to \e_0(t_0,x_0,a_0)\;\;\;\it{for every sequence}\;\;\; (t_i,x_i,a_i)\rightarrow (t_0,x_0,a_0).
\end{equation*}
\end{Th}

\begin{proof}
Because of inequality \eqref{stbprof1}, it is sufficient to show that
\begin{equation}\label{thm-zjnzz1}
\mathscr{H}(\Phi_i(t_i,x_i,a_i),\Phi_0(t_0,x_0,a_0))\to 0,\quad \forall\;(t_i,x_i,a_i)\rightarrow (t_0,x_0,a_0).
\end{equation}

Let $(t_i,x_i,a_i)\rightarrow (t_0,x_0,a_0)$. Then, by our hypotheses, we have $\omega_i(t_i,x_i)\to \omega_0(t_0,x_0)$. The latter, together with \eqref{zmkno} and \eqref{thm-zjnzz3}, implies that $\varphi_i(t_i,x_i,a_i)\to \varphi_0(t_0,x_0,a_0)$. Let $C>0$ be a constant such that  $\varphi_i(t_i,x_i,a_i)\leq C$ for every $i\in\N\cup\{0\}$. Since  $\|G_i(t_i,x_i,a_i)\|\leq \varphi_i(t_i,x_i,a_i)$ for every $i\in\N\cup\{0\}$, thus  $\Phi_i(t_i,x_i,a_i)\subset G_i(t_i,x_i,a_i)\subset\B_C$ for every $i\in\N\cup\{0\}$. Because of Lemma \ref{zmkh}, it is sufficient to show that
\begin{equation}\label{thm-zjnzz2}
\lim_{i\to\infty}\Phi_i(t_i,x_i,a_i)=\Phi_0(t_0,x_0,a_0).
\end{equation}

By the inequality \eqref{dow-omk} for all $i\in\N$ we have 
\begin{equation*}
\mathscr{H}(G_i(t_i,x_i,a_i),G_0(t_0,x_0,a_0))\leq |\varphi_i(t_i,x_i,a_i)-\varphi_0(t_0,x_0,a_0)|+2|\omega_i(t_i,x_i)\,a_i- \omega(t_0,x_0)\,a_0|.
\end{equation*}
Passing to the limit, we obtain $\lim_{i\to\infty}\mathscr{H}(G_i(t_i,x_i,a_i),G_0(t_0,x_0,a_0))=0$.
Therefore, by Lemma \ref{zmkh} we have
\begin{equation}\label{thm-zjnzz4}
    \lim_{i\to\infty}G_i(t_i,x_i,a_i)=G_0(t_0,x_0,a_0).
\end{equation}

Let $\I G_0(t_0,x_0,a_0)\not=\emptyset$. Then $E_{L_0}(t_0,x_0)\cap\I G_0(t_0,x_0,a_0)\not=\emptyset$. Thus, by Theorem~\ref{thmozkp} and properties \eqref{thm-zjnzz3}, \eqref{thm-zjnzz4}, we have $\lim_{i\to\infty}\Phi_i(t_i,x_i,a_i)=\Phi_0(t_0,x_0,a_0)$.

Let $\I G_0(t_0,x_0,a_0)=\emptyset$. Then $G_0(t_0,x_0,a_0)=\Phi_0(t_0,x_0,a_0)=\{\omega_0(t_0,x_0)a_0\}\!\subset\! E_{L_0}(t_0,x_0)$. Because of  \eqref{thm-zjnzz4},
$\limsup_{i\to\infty}\Phi_i(t_i,x_i,a_i)\subset \{\omega(t_0,x_0)\,a_0\}$.
Let $y_i\in \Phi_i(t_i,x_i,a_i)$ for every $i\in\N$. Then $y_i\in G_i(t_i,x_i,a_i)$ for all $i\in\N$. Therefore, by definition of $G_i(t_i,x_i,a_i)$ we have $|y_i-\omega_i(t_i,x_i)\,a_i|\leq 2d(\omega_i(t_i,x_i)\,a_i,E_{L_i}(t_i,x_i))$ for all $i\in\N$. Because of \eqref{zmkno} and \eqref{thm-zjnzz3},
$$\lim_{i\to\infty}d(\omega_i(t_i,x_i)\,a_i,E_{L_i}(t_i,x_i))=d(\omega_0(t_0,x_0)\,a,E_{L_0}(t_0,x_0))=0.$$ 
Thus, $y_i\to \omega_0(t_0,x_0)\,a_0$. It means that $\omega(t_0,x_0)\,a_0\in\liminf_{i\to\infty}\Phi_i(t_i,x_i,a_i)$. Consequently,
\begin{equation*}
\{\omega(t_0,x_0)\,a_0\}\subset\liminf_{i\to\infty}\Phi_i(t_i,x_i,a_i)\subset\limsup_{i\to\infty}\Phi_i(t_i,x_i,a_i)\subset\{\omega(t_0,x_0)\,a_0\}.
\end{equation*}
So, $\lim_{i\to\infty}\Phi_i(t_i,x_i,a_i)=\{\omega_0(t_0,x_0)\,a_0\}=\Phi_0(t_0,x_0,a_0)$, that completes the proof.
\end{proof}

\begin{Rem}\label{rem-stb-001}
Let $\e_i:[0,T]\times\R^n\times\R^{n+1}\rightarrow\R^{n+1}$, $i\in\N\cup\{0\}$, be as above. For all $i\in\N\cup\{0\}$ we define the functions $f_i:[0,T]\times\R^n\times \R^{n+1}\rightarrow\R^n$ and $l_i:[0,T]\times\R^n\times \R^{n+1}\rightarrow\R$ by 
\begin{equation*}
f_i(t,x,a):=\pi_1(\e_i(t,x,a))\;\;\;\te{and}\;\;\; l_i(t,x,a):=\pi_2(\e_i(t,x,a)),
\end{equation*}
where $\pi_1(v,\eta)=v$ and $\pi_2(v,\eta)=\eta$ for all $v\in\R^n$ and $\eta\in\R$.
Then for all $t\in[0,T]$, $x\in\R^n$, $a\in \R^{n+1}$,  $i\in\N\cup\{0\}$ the following
equality holds:
\begin{equation*}
\e_i(t,x,a)=(f_i(t,x,a),l_i(t,x,a)).
\end{equation*}
Therefore, for all $i\in\N$ we obtain:
\begin{eqnarray*}
|f_i(t_i,x_i,a_i)-f_0(t_0,x_0,a_0)| &\!\!\!\!\leq\!\!\!\! &  |\e_i(t_i,x_i,a_i)-\e_0(t_0,x_0,a_0)|,\\
|\,l_i(t_i,x_i,a_i)-l_0(t_0,x_0,a_0)\,| &\!\!\!\!\leq\!\!\!\! & |\e_i(t_i,x_i,a_i)-\e_0(t_0,x_0,a_0)|.
\end{eqnarray*}
\end{Rem}

\begin{Rem}
Theorem \ref{thm-zjnzz} and Remark \ref{rem-stb-001} imply Theorem \ref{thm-rep-stab2}, if in the place of  $\omega_i(t,x)$ we take $\omega_i(t,x):=|\lambda_i(t,x)|+|H_i(t,x,0)|+c_i(t)(1+|x|)+1$ for all $i\in\N\cup\{0\}$. Theorem \ref{thm-rep-stab4} can be proven similarly as above, indeed, it is enough to fix $t\in[0,T]$.
\end{Rem}


\section{Proof of Theorem \ref{thm-reduct}}\label{thm-reduct-sect}

\noindent Let $\|x\|:=\sup\,\{\,|x(t)|\,\mid t\in[0,1]\,\}$. Because of  $L(t,x,\cdot\,)=H^{\ast}(t,x,\cdot\,)$ we obtain  $-|H(t,x,0)|\leq L(t,x,v)$ for all $t\in[0,T]$, $x\in\R^n$, $v\in\R^n$. Moreover, if $H(t,\cdot,0)$ is $k_R(t)$-Lipschitz on $\B_R$ for all $t\in[0,1]$ and $R>0$, then for every $x(\cdot)\in \mathcal{A}([0,1],\R^n)$ we have
\begin{equation}\label{roz6-r1}
-k_{\|x\|}(t)\|x\|-|H(t,0,0)|\,\leq\, L(t,x(t),\dot{x}(t)) \;\;a.e.\;t\in[0,1].
\end{equation}

\begin{Lem}\label{roz6-l1}
Assume that \te{(H1)$-$(H4)} and \te{(HLC)} hold with integrable functions $c(\cdot)$, $k_R(\cdot)$, $H(\cdot,0,0)$.  Assume further that $\phi$ is a proper, lower semicontinuous function  and  there exists $M\geq 0$ such that $\min\{\,|z|,|x|\,\}\leq M$ for all $(z,x)\in\D\phi$. Then there exist $D,R\geq 0$ such that for all $x(\cdot)\in \mathcal{A}([0,1],\R^n)$ we have
\begin{equation}\label{roz6-nl1}
-D-R\int_0^1k_R(t)\,dt-\int_0^1|H(t,0,0)|\,dt\,\leq\,\Gamma[x(\cdot)].
\end{equation}
\end{Lem}

\begin{proof}
Our assumptions and inequality \eqref{roz6-r1} imply that the functional $\Gamma[\cdot]$ is well defined  and $-\infty<\Gamma[x(\cdot)]\leq +\infty$. Without loss of generality we can assume that $\Gamma[x(\cdot)]<+\infty$  for some arc $x(\cdot)\in \mathcal{A}([0,1],\R^n)$. Then we have
\begin{equation*}
\phi(x(0),x(1))<+\infty\quad\tn{and}\quad\int_{0}^1L(t,x(t),\dot{x}(t))\,dt<+\infty.
\end{equation*}
Therefore, $(x(0),x(1))\in\D\phi$ and $\dot{x}(t)\in\D L(t,x(t),\cdot)$ for a.e. $t\in[0,1]$. The latter, together with our assumptions, implies that $\min\{\,|x(0)|,|x(1)|\,\}\leq M$ and $c(t)(1+|x(t)|)\geq |\dot{x}(t)|$ for a.e. $t\in[0,1]$. Therefore, because of Gronwall’s Lemma,
\begin{equation*}
\|x\|\leq\left(M+\int_0^1c(t)\,dt\right)\,\exp\left(\int_0^1c(t)\,dt\right)=:R.
\end{equation*} 
Since $\phi$ is proper, lower semicontinuous function, there exists $D>0$ such that $-D\leq\phi(z,x)$ for all $z,x\in\B_R$. From the above and  \eqref{roz6-r1} we obtain the inequality \eqref{roz6-nl1}.
\end{proof}

 If the triple $(B,f,l)$ is a representation of $H$, then $-|H(t,x,0)|\leq l(t,x,a)$ for all $t\in[0,T]$, $x\in\R^n$, $a\in\B$. Moreover, if $H(t,\cdot,0)$ is $k_R(t)$-Lipschitz on $\B_R$ for all $t\in[0,1]$ and $R>0$, then for all $(x,a)(\cdot)\in \mathcal{A}([0,1],\R^n)\times L^1([0,1],\B)$  we have
\begin{equation}\label{roz6-r2}
-k_{\|x\|}(t)\|x\|-|H(t,0,0)|\,\leq\, l(t,x(t),a(t)) \;\;\te{for all}\;\;t\in[0,1].
\end{equation}

\begin{Lem}\label{roz6-l2}
Assume that \te{(H1)$-$(H4)}, \te{(HLC)} and \te{(BLC)} hold with integrable functions $c(\cdot)$, $k_R(\cdot)$, $H(\cdot,0,0)$, $\lambda(\cdot,0)$.  We consider the representation $(\B,f,l)$ of $H$ defined as in Theorem~\ref{th-rprez-glo12}.  Assume further that $\phi$ is a proper, lower semicontinuous function  and  there exists $M\geq 0$ such that $\min\{\,|z|,|x|\,\}\leq M$ for all $(z,x)\in\D\phi$. Then there exist $D,R\geq 0$ such that for all $(x,a)(\cdot)\!\in \!\mathcal{A}([0,1],\R^n)\times L^1([0,1],\B)$ satisfying $\dot{x}(t)\!\!=\!\!f(t,x(t),a(t))$ we have
\begin{equation}\label{roz6-nl2}
-D-R\int_0^1k_R(t)\,dt-\int_0^1|H(t,0,0)|\,dt\,\leq\,\Lambda[(x,a)(\cdot)].
\end{equation}
\end{Lem}

\begin{proof}
Our assumptions and inequality \eqref{roz6-r2}  imply that the functional $\Lambda[\cdot]$ is well defined and $-\infty<\Lambda[(x,a)(\cdot)]\leq+\infty$.  Without loss of generality we can assume that $\Lambda[(x,a)(\cdot)]<+\infty$  for some arc $(x,a)(\cdot)\in \mathcal{A}([0,1],\R^n)\times L^1([0,1],\B)$ satisfying $\dot{x}(t)=f(t,x(t),a(t))$ for a.e. $t\in[0,1]$. The latter, together with our assumptions, implies that $\min\{\,|x(0)|,|x(1)|\,\}\leq M$ and $|\dot{x}(t)|\leq c(t)(1+|x(t)|)$ for a.e. $t\in[0,1]$. Therefore, because of Gronwall’s Lemma,
\begin{equation*}
\|x\|\leq\left(M+\int_0^1c(t)\,dt\right)\,\exp\left(\int_0^1c(t)\,dt\right)=:R.
\end{equation*} 
Since $\phi$ is a proper, lower semicontinuous function, there exists $D>0$ such that $-D\leq\phi(z,x)$ for all $z,x\in\B_R$. From the above and   \eqref{roz6-r2} we obtain the inequality \eqref{roz6-nl2}.
\end{proof}

\begin{Rem}
Our assumptions and inequalities \eqref{roz6-r1} and \eqref{roz6-nl1} imply that the functional $\Gamma[\cdot]$ is well defined and $-\infty<\inf\Gamma[x(\cdot)]$. Similarly, our assumptions and inequalities \eqref{roz6-r2} and \eqref{roz6-nl2} imply that the functional $\Lambda[\cdot]$ is well defined and $-\infty<\inf\Lambda[(x,a)(\cdot)]$.
\end{Rem}

\begin{Th}\label{roz6-thm1}
Assume that \te{(H1)$-$(H4)}, \te{(HLC)} and \te{(BLC)} hold with integrable functions $c(\cdot)$, $k_R(\cdot)$, $H(\cdot,0,0)$, $\lambda(\cdot,0)$.  We consider the representation $(\B,f,l)$ of $H$ defined as in Theorem~\ref{th-rprez-glo12}. Assume further that $\phi$ is a proper, lower semicontinuous function  and  there exists $M\geq 0$ such that $\min\{\,|z|,|x|\,\}\leq M$ for all $(z,x)\in\D\phi$. Then
\begin{equation}\label{roz6-rthm1}
\inf\Gamma[x(\cdot)]\,=\,\inf\Lambda[(x,a)(\cdot)].
\end{equation}
\end{Th}

\begin{proof}
We start with the proof of the inequality:
\begin{equation}\label{roz6-rthm011}
\inf\Gamma[x(\cdot)]\,\geq\,\inf\Lambda[(x,a)(\cdot)].
\end{equation}
Without loss of generality we can assume that $-\infty<\inf\Gamma[x(\cdot)]<+\infty$. Let us fix  $\varepsilon>0$. Then there exists  $\kr{x}(\cdot)\in \mathcal{A}([0,1],\R^n)$ such that $\inf\Gamma[x(\cdot)]+\varepsilon\,\geq\,\Gamma[\kr{x}(\cdot)]$. We define $\kr{u}(\cdot)\in \mathcal{A}([0,1],\R)$ by the formula
$$\kr{u}(t):=\int_0^tL(s,\kr{x}(s),\dot{\kr{x}}(s))\,ds.$$
We notice that $\dot{\kr{u}}(t)=L(t,\kr{x}(t),\dot{\kr{x}}(t))$  for a.e. $t\in[0,1]$. Therefore, $(\dot{\kr{x}}(t),\dot{\kr{u}}(t))\in\G L(t,\kr{x}(t),\cdot)$ for a.e. $t\in[0,1]$. By (A3) of Theorem \ref{th-rprez-glo12}, $\G L(t,\kr{x}(t),\cdot)\subset\e(t,\kr{x}(t),\B)$ for all $t\in[0,1]$, where $\e(t,x,\cdot)=(f(t,x,\cdot),l(t,x,\cdot))$. From the above and \cite[Thm. 8.2.10]{A-F} there exists a measurable function $\kr{a}(\cdot)$ defined on $[0,1]$ with values in $\B$ such that  for a.e. $t\in[0,1]$,
$$(\dot{\kr{x}}(t),\dot{\kr{u}}(t))=\e(t,\kr{x}(t),\kr{a}(t))=(f(t,\kr{x}(t),\kr{a}(t)),l(t,\kr{x}(t),\kr{a}(t))).$$
Consequently, we have
\vspace{-3mm}
\begin{eqnarray*}
\inf\Gamma[x(\cdot)]+\varepsilon &\geq & \Gamma[\kr{x}(\cdot)]\;\;=\;\;\phi(\kr{x}(0),\kr{x}(1))+\int_0^1L(t,\kr{x}(t),\dot{\kr{x}}(t))\,dt\\
&=& \phi(\kr{x}(0),\kr{x}(1))+\int_0^1\dot{\kr{u}}(t)\,dt\\
&=& \phi(\kr{x}(0),\kr{x}(1))+\int_0^1l(t,\kr{x}(t),\kr{a}(t))\,dt\\[2mm]
&=& \Lambda[(\kr{x},\kr{a})(\cdot)]\;\;\geq\;\; \inf\Lambda[(x,a)(\cdot)].
\end{eqnarray*}
Therefore, $\inf\Gamma[x(\cdot)]+\varepsilon\geq\inf\Lambda[(x,a)(\cdot)]$.  The latter inequality, together with the arbitrariness of $\varepsilon>0$, implies  \eqref{roz6-rthm011}.

Now we prove the inequality:
\begin{equation}\label{roz6-rthm021}
\inf\Gamma[x(\cdot)]\,\leq\,\inf\Lambda[(x,a)(\cdot)].
\end{equation}
Without loss of generality we can assume that $-\infty<\inf\Lambda[(x,a)(\cdot)]<+\infty$. Let us fix  $\varepsilon>0$. Then there exists $(\kr{x},\kr{a})(\cdot)\in \mathcal{A}([0,1],\R^n)\times L^1([0,1],\B)$ satisfying $\dot{\kr{x}}(t)=f(t,\kr{x}(t),\kr{a}(t))$ for a.e. $t\in[0,1]$ and $\inf\Lambda[(x,a)(\cdot)]+\varepsilon\geq\Lambda[(\kr{x},\kr{a})(\cdot)]$. From our assumptions and Lemma \ref{lem-tran-r1},
$$L(t,\kr{x}(t),f(t,\kr{x}(t),\kr{a}(t)))\leq l(t,\kr{x}(t),\kr{a}(t))\;\;\tn{for all}\;\;t\in[0,1].$$
From the above, $L(t,\kr{x}(t),\dot{\kr{x}}(t))\leq l(t,\kr{x}(t),\kr{a}(t))$ for a.e. $t\in[0,1]$.  Consequently, we have
\begin{eqnarray*}
\inf\Lambda[(x,a)(\cdot)]+\varepsilon &\geq & \Lambda[(\kr{x},\kr{a})(\cdot)]\\
&=& \phi(\kr{x}(0),\kr{x}(1))+\int_0^1l(t,\kr{x}(t),\kr{a}(t))\,dt\\
&\geq & \phi(\kr{x}(0),\kr{x}(1))+\int_0^1L(t,\kr{x}(t),\dot{\kr{x}}(t))\,dt\\[2mm]
&=& \Gamma[\kr{x}(\cdot)]\;\;\geq\;\; \inf\Gamma[x(\cdot)].
\end{eqnarray*}
Therefore, $\inf\Lambda[(x,a)(\cdot)]+\varepsilon\geq\inf\Gamma[x(\cdot)]$.  The latter inequality, together with the arbitrariness of $\varepsilon>0$, implies  \eqref{roz6-rthm021}.

Combining inequalities \eqref{roz6-rthm011} and \eqref{roz6-rthm021} we obtain the equality \eqref{roz6-rthm1}. 
\end{proof}

\begin{Rem}\label{roz6-rem1}
From the equality \eqref{roz6-rthm1} and its proof it follows that if $\kr{x}(\cdot)$ is the optimal arc of $(\mathcal{P}_{v})$ such that $\kr{x}(\cdot)\in\D\Gamma$, then there exists  $\kr{a}(\cdot)$ such that $(\kr{x},\kr{a})(\cdot)$ is also the optimal arc of $(\mathcal{P}_{c})$ and $(\kr{x},\kr{a})(\cdot)\in\D\Lambda$;  conversely, if $(\kr{x},\kr{a})(\cdot)$ is the optimal arc of $(\mathcal{P}_{c})$, then $\kr{x}(\cdot)$ is also the optimal arc of $(\mathcal{P}_{v})$.
\end{Rem}

\begin{Th}\label{roz6-thm2}
Assume that \te{(H1)$-$(H4)} and \te{(HLC)} hold with integrable functions $c(\cdot)$, $k_R(\cdot)$, $H(\cdot,0,0)$.  Assume further that $\phi$ is a proper, lower semicontinuous function  and  there exists $M\geq 0$ such that $\min\{\,|z|,|x|\,\}\leq M$ for all $(z,x)\in\D\phi$. Then there exists the optimal arc $\kr{x}(\cdot)$ of the variational problem  $(\mathcal{P}_{v})$.
\end{Th}

\begin{proof}
Without loss of generality we can assume that  $-\infty<\inf\Gamma[x(\cdot)]<+\infty$. Then there exists  a sequence $\{\,x_i(\cdot)\,\}\subset\mathcal{A}([0,1],\R^n)$ such that for all $i\in\N$ we have
\begin{equation*}
\inf\Gamma[x(\cdot)]\leq\Gamma[x_i(\cdot)]\leq\inf\Gamma[x(\cdot)]+1/i.
\end{equation*}
Since $\Gamma[x_i(\cdot)]<+\infty$ for all $i\in\N$, we  deduce that for all $i\in\N$
\begin{equation*}
\phi(x_i(0),x_i(1))<+\infty\quad\tn{and}\quad\int_{0}^1L(t,x_i(t),\dot{x}_i(t))\,dt<+\infty.
\end{equation*}
Therefore, $(x_i(0),x_i(1))\in\D\phi$ and $\dot{x}_i(t)\in\D L(t,x_i(t),\cdot)$, for a.e. $t\in[0,1]$ and for all $i\in\N$. The latter, together with our assumptions, implies that $\min\{\,|x_i(0)|,|x_i(1)|\,\}\leq M$ and $|\dot{x}_i(t)|\leq c(t)(1+|x_i(t)|)$ for a.e. $t\in[0,1]$ and for all $i\in\N$. Therefore, because of Gronwall’s Lemma, for all $i\in\N$
\begin{equation*}
\|x_i\|\leq\left(M+\int_0^1c(t)\,dt\right)\,\exp\left(\int_0^1c(t)\,dt\right)=:R.
\end{equation*} 
From the above we obtain $\|x_i\|\leq R$ and $|\dot{x}_i(t)|\leq c(t)(1+R)$ for a.e. $t\in[0,1]$ and for all $i\in\N$. Because of Arzel\`a-Ascoli and Dunford-Pettis Theorems, there exists a subsequence (denoted again by $\{x_i\}$) such that $x_i\rightrightarrows\kr{x}$ and $\dot{x}_i\rightharpoonup\dot{\kr{x}}$ in $L^1([0,1],\R^n)$. Therefore, because of \cite[Semicontinuity Theorem]{RTR},
\begin{equation*}
\liminf_{i\to\infty}\int_0^1L(t,x_i(t),\dot{x}_i(t))\,dt\;\geq\,\int_0^1L(t,\kr{x}(t),\dot{\kr{x}}(t))\,dt.
\end{equation*}
Consequently, we have
\begin{eqnarray*}
\inf\Gamma[x(\cdot)] &=& \lim_{i\to\infty}\Gamma[x_i(\cdot)]\\
&\geq & \liminf_{i\to\infty}\phi(x_i(0),x_i(1))+\liminf_{i\to\infty}\int_0^1L(t,x_i(t),\dot{x}_i(t))\,dt\\
&\geq & \phi(\kr{x}(0),\kr{x}(1))+\int_0^1L(t,\kr{x}(t),\dot{\kr{x}}(t))\,dt\\[2mm]
&=& \Gamma[\kr{x}(\cdot)]\;\;\geq\;\;\inf\Gamma[x(\cdot)].
\end{eqnarray*}
Therefore, $\inf\Gamma[x(\cdot)]=\Gamma[\kr{x}(\cdot)]$. Hence, it follows that $\kr{x}(\cdot)$ is the optimal arc of the variational problem  $(\mathcal{P}_{v})$.
\end{proof}

\begin{Rem}
Theorems \ref{roz6-thm1} and \ref{roz6-thm2}, together with Remark \ref{roz6-rem1}, imply Theorem \ref{thm-reduct}.
\end{Rem}



\begin{thebibliography}{99}


\bibitem{A-F} \textsc{J.-P. Aubin, H. Frankowska}, {\em Set-Valued Analysis}, Birkh\"{a}user, 1990, (Modern Birkh\"{a}user Classics, reprint 2008).

\bibitem{B-CD} \textsc{M. Bardi, I. Capuzzo-Dolcetta}, {\em Optimal control and viscosity solutions of Hamilton-Jacobi-Bellman equations}, Birk\-h\"{a}user, Boston 1997.

\bibitem{C-S-2004} \textsc{P. Cannarsa, C. Sinestrari}, {\em Semiconcave Functions, Hamilton-Jacobi Equations, and Optimal Control}, in: Progress in Nonlinear Differential Equations
and their Applications, vol. 58, Birkhäuser, Boston, MA, 2004.




\bibitem{FC} \textsc{F.H. Clarke}, {\em Optimization and nonsmooth analysis}, New York: Wiley, 1983.

\bibitem{HF} \textsc{H. Frankowska}, {\em Lower semicontinuous solutions of Hamilton-Jacobi-Bellman equations}, SIAM J. Control Optim., {31} (1993), 257--272.

\bibitem{F-P-Rz} \textsc{H. Frankowska, S.Plaskacz, T.Rzeżuchowski}, {\em Measurable viability theorems and Hamilton-Jacobi-Bellman equation}, J. Differential Equations, {116} (1995), 265--305.


\bibitem{F-S} \textsc{H. Frankowska, H. Sedrakyan},  {\em Stable representation of convex Hamiltonians}, Nonlinear Anal., {100} (2014), 30--42.



\bibitem{HI} \textsc{H. Ishii}, in: K. Masuda, M. Mimura (Eds.), {\em On Representations of Solutions of Hamilton-Jacobi Equations with Convex Hamiltonians}, in: Recent Topics
in Nonlinear PDE, II, vol. 128, North Holland, Amsterdam, (1985), 15--52.


\bibitem{AM} \textsc{A. Misztela}, {\em Representation of Hamilton-Jacobi equation in optimal control theory with unbounded control set}, arXiv:1807.03640.


\bibitem{CO-69}\textsc{ C. Olech},  {\em Existence theorems for optimal problems with vector-valued cost functions}, Trans. Amer. Math. Soc., 136 (1969), 159-180.

\bibitem{FR} \textsc{F. Rampazzo},  {\em Faithful representations for convex Hamilton-Jacobi equations}, SIAM J. Control Optim., {44(3)} (2005), 867--884.


\bibitem{RTR73} \textsc{R.T. Rockafellar}, {\em Optimal arcs and the minimum value function in problems of Lagrange}, Trans. Amer. Math. Soc., {180} (1973),  53--84.

\bibitem{RTR} \textsc{R.T. Rockafellar}, {\em Existence theorems for general control problems of Bolza and Lagrange}, Adv. Math., {15} (1975),  312--333.

\bibitem{R-W} \textsc{R.T. Rockafellar, R. J.-B. Wets}, {\em Variational Analysis}, Springer-Verlag, Berlin 1998.

\bibitem{HS} \textsc{H. Sedrakyan},  {\em Stability of solutions to Hamilton-Jacobi equations
under state constraints},  J. Optim. Theory Appl., {168} (2016), 63--91. 


\end{thebibliography}
\end{document}